\def\isFV{true}
\def\isBJ{true}
\documentclass{birkjour}
\newcommand{\fvonly}[1]{\ifthenelse{\equal{\isFV}{true}}{#1}{}}
\newcommand{\fvsh}[2]{\ifthenelse{\equal{\isFV}{true}}{#1}{#2}}
\newcommand{\mydisp}[2]{\ifthenelse{\equal{\isFV}{true}}{\[ #1\]}{$ #1 $#2}}
\newcommand{\vsp}{V\xspace}
 
\def\isBeamer{false}

\usepackage{xspace}
\usepackage{helvet}
\usepackage{courier}
\usepackage{fancyhdr}
\usepackage{type1cm}         
\usepackage{subeqnarray}
\usepackage{graphicx}        
\usepackage{multicol}        
\usepackage[bottom]{footmisc}
\usepackage{amscd}	
\usepackage{enumerate}
\usepackage{amsmath, amsthm, latexsym}
\usepackage{array}
\usepackage{tabularx}
\usepackage{url}
\usepackage{relsize}
\usepackage{ifthen}
\usepackage{xcolor}
\usepackage[newitem,newenum]{paralist}
 
\def\isBeamer{false}

\newcommand{\mychapterskip}{\ifthenelse{\equal{\isMPK}{true}}
{\vspace{-.8in}}{}}

\newcommand{\fullversioncolor}{\color{black!50!blue}}
\newcommand{\versions}[2]{\ifthenelse{\equal{\isfullversion}{true}}{{\fullversioncolor#1}}{#2}}

\renewcommand{\vec}[1]{\mathbf{#1}}
\newcommand{\quot}[1]{``#1''}

\newcommand{\R}[1]{\ensuremath{\mathbb{R}^{#1}\xspace}}
\newcommand{\Euc}[1]{\ensuremath{\mathbf{E}^{#1}\xspace}}

\newcommand{\RD}[1]{\ensuremath{(\mathbb{R}^{#1})^{*}\xspace}}
\newcommand\proj[1]{\ensuremath{\mathbf{P}(#1)}\xspace}
\newcommand\RP[1]{\ensuremath{\mathbb{R}{P^{#1}}\xspace}}

\newcommand{\e}[1]{\vec{e}_{#1}}
\newcommand{\EE}[1]{\vec{E}_{#1}}
\newcommand{\one}{\vec{1}}
\newcommand{\eye}{\vec{I}}

\newcommand{\MN}{metric-neutral\xspace}

\newcolumntype{Y}{X}

\newcommand{\grade}[2]{\langle #1 \rangle_{#2}}

\newcommand{\pdgrass}[1]{\proj{\bigwedge({\mathbb{R}^{#1})^*}}}

\newcommand{\clal}[3]{\mathbb{R}_{#1,#2,#3}\xspace}
\newcommand{\dclal}[3]{\mathbb{R}^*_{#1,#2,#3}\xspace}

\newcommand{\pclal}[3]{\proj{\mathbb{R}_{#1,#2,#3}}\xspace}
\newcommand{\pdclal}[3]{\proj{\mathbb{R}^*_{#1,#2,#3}}\xspace}

\newcommand{\myexercise}{\hspace{-.31in}\textbf{Exercise:~~\xspace}}
\newcommand{\myexercises}{\hspace{-.31in}\textbf{Exercises:~~\xspace}}
\newcommand{\myboldhead}[1]{\vspace{.1in}\hspace{-.35in}\textbf{#1.}}

\newcommand{\app}[1]{\Sec{sec:J}}

\definecolor{mypurple}{rgb}{1,0,1}

\definecolor{mygreen}{rgb}{0, .5, 0}

\newcommand{\Fig}[1]{Fig.~\ref{#1}}
\newcommand{\Tab}[1]{Table~\ref{#1}}
\newcommand{\Sec}[1]{Sect.~\ref{#1}}

\ifthenelse{\equal{\isBeamer}{true}}
{

}
{

}

\newcommand{\mycorrection}[1]{}

\newtheorem{mythm}{Theorem}
\newtheorem{mylemma}{Lemma}

\fvsh{
}{}
\ifthenelse{\equal{\isBJ}{true}}{\def\xyx{-.32in}}{\def\xyx{-.23in}}
\renewcommand{\myboldhead}[1]{\vspace{.1in}\hspace{\xyx}\textbf{#1.}}

\usepackage{graphicx}
\usepackage{wrapfig}
\usepackage{amssymb}
\usepackage[page,title]{appendix}

\graphicspath{{.}{./pictures/}} 
\newcommand{\gTh}{\cite{gunnThesis}\xspace}

\begin{document}

\ifthenelse{\equal{\isBJ}{false}}{

\title{Doing euclidean plane geometry\\using projective geometric algebra}
\author{Charles G. Gunn } 
}
{
\title{Doing euclidean plane geometry\\using projective geometric algebra\footnote{This article has been published as \cite{gunn2016b}, DOI 10.1007/s00006-016-0731-5.  The final publication is available at link.springer.com.}}
\author{Charles G. Gunn}
\address{ Instit\"{u}t f\"{u}r Mathematik  MA 8-3\\
Technische Universit\"{a}t Berlin\\
Str. des 17 Juni 136\\
10623 Berlin Germany\\
\\
Charles G. Gunn\\Raum+Gegenraum \\  Brieselanger Weg 1 \\ 14612 Falkensee Germany  }
\email{gunn@math.tu-berlin.de}
\keywords{euclidean geometry, plane geometry, geometric algebra, projective geometric algebra, degenerate signature,  sandwich operator, orthogonal projection, isometry}
}
\maketitle

\begin{abstract}
\fvsh{
The article presents a new approach to euclidean plane geometry based on  projective geometric algebra (PGA).  It is designed for anyone with an interest in plane geometry, or who wishes to familiarize themselves with PGA.  After a brief review of PGA, the article focuses on $\pdclal{2}{0}{1}$, the PGA for euclidean plane geometry.  It first explores the geometric product involving pairs and triples of basic elements (points and lines), establishing a wealth of fundamental metric and non-metric properties.  It then applies the algebra to a variety of familiar topics in plane euclidean geometry and shows that it compares favorably with other approaches in regard to completeness, compactness, practicality, and elegance. The seamless integration of euclidean and ideal (or \quot{infinite}) elements forms an essential and novel feature of the treatment.  Numerous figures accompany the text.   For readers with the requisite mathematical background, a self-contained coordinate-free introduction to the algebra is provided in an appendix.  
}
{
The article provides an introduction to the use of projective geometric algebra (PGA) for doing euclidean plane geometry. After a quick review of PGA, the article undertakes a detailed study of the geometric product of blades: pairs of lines, pairs of points, point-line pair, 3 lines, and 3 points.   It then applies these results to a selection of  application areas:  distance formulae, isometries via sandwiches, 
sums and differences of points and of lines,  orthogonal projections, 
and a step-by-step solution of a sample geometric construction.   It concludes by briefly comparing  PGA to alternative approaches.
}
\end{abstract}

\section{Introduction}

The $19^{th}$ century witnessed an unprecedented development of geometry and algebra.  We need only mention the development of projective and non-euclidean geometries, complex and quaternion number systems, and Grassmann algebra to indicate the depth and breadth of these developments, many of which came together in William Clifford's invention of \emph{geometric algebra} (\cite{clifford78}).  This is a  comprehensive algebraic structure that models both incidence relations and metric relations -- for a variety of metric geometries -- in a concise and powerful form, and which is  ideally suited to computational implementation. The teaching and practice of euclidean  geometry in the $20^{th}$ century, however, remained largely untouched by these developments, except for the introduction of vector and linear algebra to supplement the standard tools of analytic geometry.    

In recent years, however,  geometric algebra  has found growing acceptance as a tool for euclidean geometry. Those seeking geometric algebra toolkits for doing $n$-dimensional euclidean geometry  find two popular solutions in the contemporary literature:    the  so-called \emph{vector geometric algebra}  (VGA), using $n$-dimensional coordinates (\cite{dfm07}, Ch. 10); and  \emph{conformal geometric algebra} (CGA), which uses $(n+2)$-dimensional coordinates (\cite{dfm07}, Ch. 13). \cite{gunn2011}, \cite{gunnFull2010}, and \cite{gunn2016} feature a third model, less well known than these two, which fits naturally between them:  \emph{projective geometric algebra} (or PGA for short), which uses $(n+1)$-dimensional coordinates to model $n$-dimensional metric spaces of constant curvature: euclidean, hyperbolic, and elliptic.   This article provides an introduction to euclidean PGA, by applying it to the euclidean plane $\Euc{2}$. 

\subsection{Structure of the article} 
\ifthenelse{\equal{\isFV}{false}}{
This article assumes familiarity with geometric algebra in general and with \cite{gunn2011} or \cite{gunnFull2010} in particular.  
It begins with a quick review of the dual projective geometric algebra $\pdclal{2}{0}{1}$ which forms the basis of PGA for $\Euc{2}$.  It then goes on to a detailed discussion of products of 2 or 3 $k$-vectors, with focus on the fruitful interplay of euclidean and ideal elements.  There follows a sequence of applications of this product to plane geometry:  distance formulae; isometries as sandwiches;  sums and differences of $k$-vectors; orthogonal projections; 
and a step-by-step solution to a classical geometric construction problem.  
Finally, the article compares  $\pdclal{2}{0}{1}$ to alternative approaches to doing plane  geometry.
}
{
\Sec{sec:algintro} gives a brief overview of geometric algebra. \Sec{sec:gaep} then introduces the dual projective geometric algebra $\pdclal{2}{0}{1}$ as a geometric algebra for doing euclidean plane geometry. There follows a discussion of the basis elements in different grades and how they can be normalized,  along with the distinction between euclidean and ideal elements. \Sec{sec:gpdetail}  examines in detail the geometric product of 2  elements of various grades and types, while \Sec{sec:threes} does the same for 3-way products.   In the following sections, the resulting  compact and powerful geometric toolkit  is  applied to a sequence of topics in plane geometry:  distance formulae (\Sec{sec:daf}), sums and differences of $k$-vectors (\Sec{sec:sumdiff}), isometries as sandwiches (\Sec{sec:isom}), 
orthogonal projections (\Sec{sec:orthpro}),  
and a step-by-step solution to a classical geometric construction problem (\Sec{sec:example}). 
\Sec{sec:neg}  gives the interested reader an overview of directions for further study.
The article concludes (\Sec{sec:eac}) by evaluating the results obtained and comparing them to alternative approaches to doing euclidean plane  geometry.  Appendix A features a coordinate-free derivation of the results of \Sec{sec:gaep} for readers with the necessary mathematical sophistication.
}
 \vspace{-.05in}
 \section{Geometric algebra fundamentals}
\label{sec:algintro}
\fvsh{ 
A self-contained introduction to geometric algebra lies outside the scope of this article.  We sketch here the essential ingredients;  interested readers are referred to  the textbook \cite{dfm07} for a modern computer science approach or \cite{artin57} for an older, more mathematical approach.  The Wikipedia article entitled \quot{geometric algebra} is also quite useful.  Readers should keep in mind that none of these references deal with degenerate metrics, which form a key feature of the approach described here.

\myboldhead{Grassmann algebra} Geometric algebra can be built upon the combination of an outer and an inner product on a vector space.  We assume the reader is familiar with real vector spaces, and also with the exterior (or Grassmann) algebra $\bigwedge{\vec{V}}$ constructed atop a real $n$-dimensional vector space $\vec{V}$.  This is a graded algebra in which the  elements of grade-$k$ ($\bigwedge^{k}{\vec{V}}$) correspond to the weighted vector subspaces of $\vec{V}$ of dimension $k$\footnote{Two elements $\vec{a}$ and $\vec{b}$ that satisfy $\vec{a} = \lambda\vec{b}$ for some non-zero $\lambda \in \mathbb{R}$ represent the same subspace, but with different weights.  The weight is discussed in more detail below in \Sec{secwandf}} .  Each grade is a vector space in its own right of dimension $n \choose k$.  The exterior product \[\wedge: \bigwedge^{k}{\vec{V}} \times \bigwedge^{m}{\vec{V}} \rightarrow \bigwedge^{k+m}{\vec{V}}\] is a binary operator that is bilinear and anti-symmetric in its arguments.  Geometrically, $\wedge$ is the join operator on the subspaces of $\vec{V}$: it gives the $(k+m)$-dimensional subspace spanned by its arguments, or 0 if they are linearly dependent.  It is also called the \emph{outer} product. The Grassmann algebra has $(n+1)$ non-zero grades, from $0$ (the scalars) to $n$ (the so-called pseudoscalars).  $\bigwedge{\vec{V}}$ has  total dimension $2^{n}$, as a glance at Pascal's triangle shows.

\myboldhead{Symmetric bilinear forms} We also assume the reader is familiar with symmetric bilinear forms on a vector space, which allow us to define inner products on $\vec{V}$.  Such a form B is characterized by its \emph{signature}, an integer triple $(p,n,z)$ where $p+n+z=n$.  Sylvester's Inertia Theorem asserts that there is a basis for V for which B is a diagonal matrix with $p$ $1$'s, $n$ $-1$'s, and $z$ 0's on the diagonal. If $z \neq 0$, we say the inner product is degenerate.  We will see below that the signature for euclidean geometry is degenerate.

\myboldhead{Measurement} In the standard euclidean vector space $\R{3}$,  measurement of angles between vectors $\vec{u} := (x_{u},y_{u},z_{u})$ and $\vec{v} := (x_{v},y_{v},z_{v})$ is determined by the standard euclidean inner product $\vec{u} \cdot \vec{v} := x_{u}x_{v}+y_{u}y_{v}+z_{u}z_{v}$, with signature $(3,0,0)$.    Using this inner product, one can compute the angle between vectors or between planes (elements of the dual vector space).  If $\vec{u}$ and $\vec{v}$ are two unit-length 1-vectors, then the inner product $\vec{u} \cdot \vec{v}$ is well-known to be the cosine of their angle. 

\myboldhead{Geometric product} The geometric algebra arises by supplementing the outer product with the inner product.  One defines the \emph{geometric product} on 1-vectors of $\bigwedge{\vec{V}}$ by \[ \vec{a}\vec{b} := \vec{a} \cdot \vec{b} + \vec{a} \wedge \vec{b}\] The right-hand side is the sum of a 0-vector (scalar) and 2-vector (plane through the origin). This definition can be extended to the whole Grassmann algebra, yielding an associative algebra called the geometric (or Clifford) algebra with signature $(p,n,z)$.   In the example above, we obtain  $\clal{3}{0}{0}$. For details see \cite{dfm07}.  

\myboldhead{Some terminology} The general element in a geometric algebra is called a \emph{multivector}.  For a multivector $\vec{M}$, the grade-$k$ part is written $\langle \vec{M} \rangle_{k}$, hence $\vec{M} = \sum_{k} \langle \vec{M} \rangle_{k}$.   An element of $\bigwedge^{k}{\vec{V}}$ is called a $k$-vector. A $k$-vector that is the product of $k$ 1-vectors is called a \emph{simple} $k$-vector, or a \emph{blade}. For a $k$-vector $\vec{A}$ and an $m$-vector $\vec{B}$, the dot product $\vec{A}\cdot \vec{B} := \langle \vec{A}\vec{B} \rangle_{| k-m |}$ is defined as  the lowest grade component of $\vec{A}\vec{B}$. The wedge $\vec{A} \wedge \vec{B} = \langle \vec{A} \vec{B} \rangle_{k+m}$ is, on the other hand,  the highest grade component. This is consistent with the definition of the product of two 1-vectors above.  $\widetilde{\vec{X}}$, the \emph{reversal} of a multi-vector $\vec{X}$, is obtained by reversing the order of all products involving 1-vectors. $\widetilde{\vec{X}}$ is an algebra involution, needed below in \ref{sec:prtworef}.

 \section{Geometric algebra for the euclidean plane}
\label{sec:gaep}
The above behavior for $\R{3}$ is typical of any geometric algebra with non-degenerate metric: the inner product provides the necessary information to calculate the angle or distance between the two elements.  What is the situation in the euclidean plane $\Euc{2}$? What kind of inner  product do we need to measure the angle between two euclidean lines?

\begin{figure}[t]
  \centering
    \setlength\fboxsep{0pt}\fbox{\includegraphics[width=0.6\textwidth]{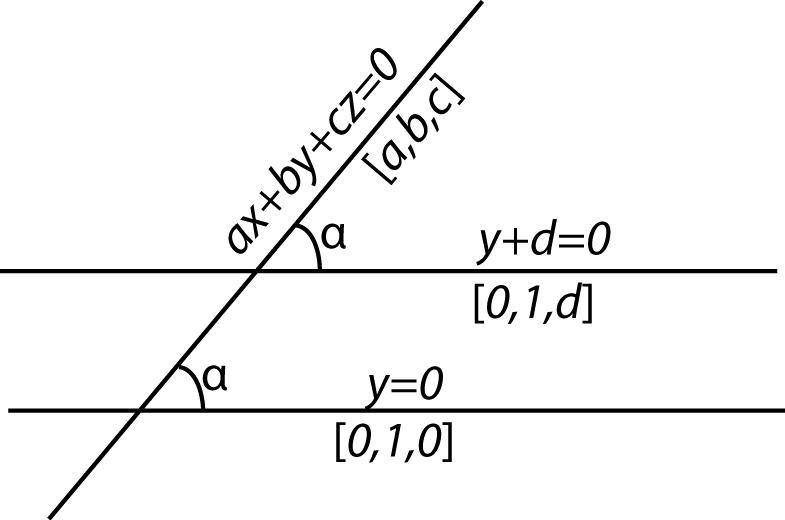}}
  \caption{Angles of euclidean lines}
\label{fig:anglelines}
\end{figure} 
Let \mydisp{
a_{0}x + b_{0}y + c_{0} = 0,~~~~~~~
a_{1}x + b_{1}y + c_{1} = 0
}{}   be two oriented lines which intersect at an angle $\alpha$.  We can assume without loss of generality that the coefficients satisfy ${a_{i}^{2} + b_{i}^{2} = 1}${.}  Then it is not difficult to show that \mydisp{a_{0} a_{1} + b_{0} b_{1} = \cos{\alpha}}{.} 
Unlike the inner product for the case of vectors in $\clal{3}{0}{0}$, here the third coordinate of the lines makes no difference in the angle calculation: translating a line changes only its third coordinate, while leaving  the angle between the lines unchanged.  Refer to Fig. \ref{fig:anglelines} which shows an example involving a general line and a pair of horizontal lines. Hence the proper signature for measuring angles in $\Euc{2}$ is $(2,0,1)$.  This is a so-called \emph{degenerate} inner product since the last entry in the signature is non-zero.

Notice that to model lines and points in a symmetric way we adopt homogeneous coordinates so line equations appear as $ax + by + cz  = 0$.  That is, we work in projective space $\RP{n}$.  Hence, to produce a geometric algebra for the euclidean plane we must attach the signature $(2,0,1)$  to a \emph{projectivized}  Grassmann algebra.  As the above discussion yields a way to measure the angle between \emph{lines} rather than the distance between \emph{points}, we choose the \emph{dual} projectivized Grassmann algebra $\pdgrass{3}$ for this purpose, where 1-vectors represents lines, 2-vectors represent points, and $\wedge$ is the \emph{meet} operator.  This leads to the geometric algebra $\pdclal{2}{0}{1}$ as the correct one for plane euclidean geometry.  We call it \emph{projective} geometric algebra  (PGA) due to its close connections to projective geometry.  (The standard Grassmann algebra leads to $\pclal{2}{0}{1}$, which models dual euclidean space, a different metric space.)

PGA for euclidean geometry first appeared in the modern literature in \cite{selig00} and \cite{selig05} and was extended and developed in \gTh, \cite{gunn2011}, \cite{gunnFull2010}, and \cite{gunn2016}. Readers unfamiliar with duality or projectivization, or just interested in a fuller, more rigorous treatment, should consult the latter references.  The  4-dimensional subalgebra consisting of scalars and bivectors, also known as the \emph{planar quaternions}, has a long history as a tool for kinematics in the plane (\cite{blaschke38}, \cite{mcc}). 

\subsection{Meet and join}

As mentioned above, the wedge operator $\wedge$ in $\pdclal{2}{0}{1}$ is the \emph{meet} operator. It is important to have access to the join operator also.   Since the typical solution to this challenge assumes a non-degenerate metric, we sketch a non-metric approach, for details see \cite{gunnThesis}.
The Poincar\'{e} isomorphism $\vec{J}: G \leftrightarrow G^{*}$ between the Grassmann algebra  $G$ and the dual Grassmann algebra $G^{*}$ can be used to define the \emph{join} operator $\vee$ in $\pdclal{2}{0}{1}$ :
\[ \vec{A} \vee \vec{B} := \vec{J}(\vec{J}(\vec{A}) \wedge \vec{J}(\vec{B})) \]
$\vec{J}$ is also sometimes called the \emph{dual coordinate} map. It is essentially an \emph{identity} map, since it maps a geometric entity in the Grassmann algebra to the \emph{same} geometric entity in the dual Grassmann algebra. For example, in projective 3-space $\RP{3}$, a line $\vec{L}$ can be represented as a bivector in G since it is the join of two points (1-vectors in $G$).  It also appears as a bivector in $G^*$ since it can also be represented as the intersection of two planes (1-vectors in $G^*$). In general, a geometric entity represented by a $k$-vector in $G$ will be represented by an $(n-k)$-vector in $G^*$, where $n$ is the dimension of the underlying vector space.  $\vec{J}$ allows one to move back and forth between these two dual representations depending on the circumstances. One can also implement the join operator using the  \emph{shuffle operator} within $\pdclal{2}{0}{1}$ (\cite{selig05}, Ch. 10).  

 }
 {
The dual projectivized geometric algebra $\pdclal{2}{0}{1}$  (\cite{selig05}, \cite{gunn2011}) is built on the \emph{dual} projectivized Grassmann algebra (where 1-vectors represent lines, not points).   Due to space limitations, we present the main features here without discussion or proof.  The inner product  signature $(2,0,1)$ has one zero term since one of the 3 coordinates of a  euclidean line has no effect on the angle it makes with other lines. See the left-hand diagram of \Fig{fig:emvl}.   The wedge operator in $\pdclal{2}{0}{1}$ is the meet operator.  To obtain the join operator $\vee$ one can use the the Poincar\'{e} isomorphism $\vec{J}: G \leftrightarrow G^{*}$ between the Grassmann algebra  $G$ and the dual Grassmann algebra $G^{*}$ (\cite{gunnThesis}):
\[ \vec{A} \vee \vec{B} := \vec{J}(\vec{J}(\vec{A}) \wedge \vec{J}(\vec{B}) \]  We present an alternative method to calculate the join specialized for $\pdclal{2}{0}{1}$ in \Sec{sec:shuffle} below.
 }
 
%

\subsection{Basis vectors of the algebra}
\label{sec:bva}
We provide here a treatment of the algebra based on a choice of basis elements; a coordinate-free treatment for more mathematically sophisticated readers can be found in Appendix A.

$\pdclal{2}{0}{1}$ has an orthogonal basis of 1-vectors  $\{\e{0}, \e{1}, \e{2}\}$ satisfying 
\mydisp{\e{0}^{2} = 0, ~~\e{1}^{2} = \e{2}^{2} = 1,~~~\e{i} \cdot \e{j} = 0~~\text{for}~~i \neq j}{.}  $\e0$ is the \emph{ideal} line of the plane (sometimes called the \quot{line at infinity}) which we  write  as $\vec{\omega}$, $\e{1}$ is the $x=0$ line (vertical!) and $\e{2}$ is the $y=0$ line (horizontal).  All lines except $\vec{\omega}$ belong to the euclidean plane and are called \emph{euclidean} lines. 
\begin{figure}[h] 
  \centering
    \setlength\fboxsep{0pt}\fbox{\includegraphics[width=0.65\textwidth]{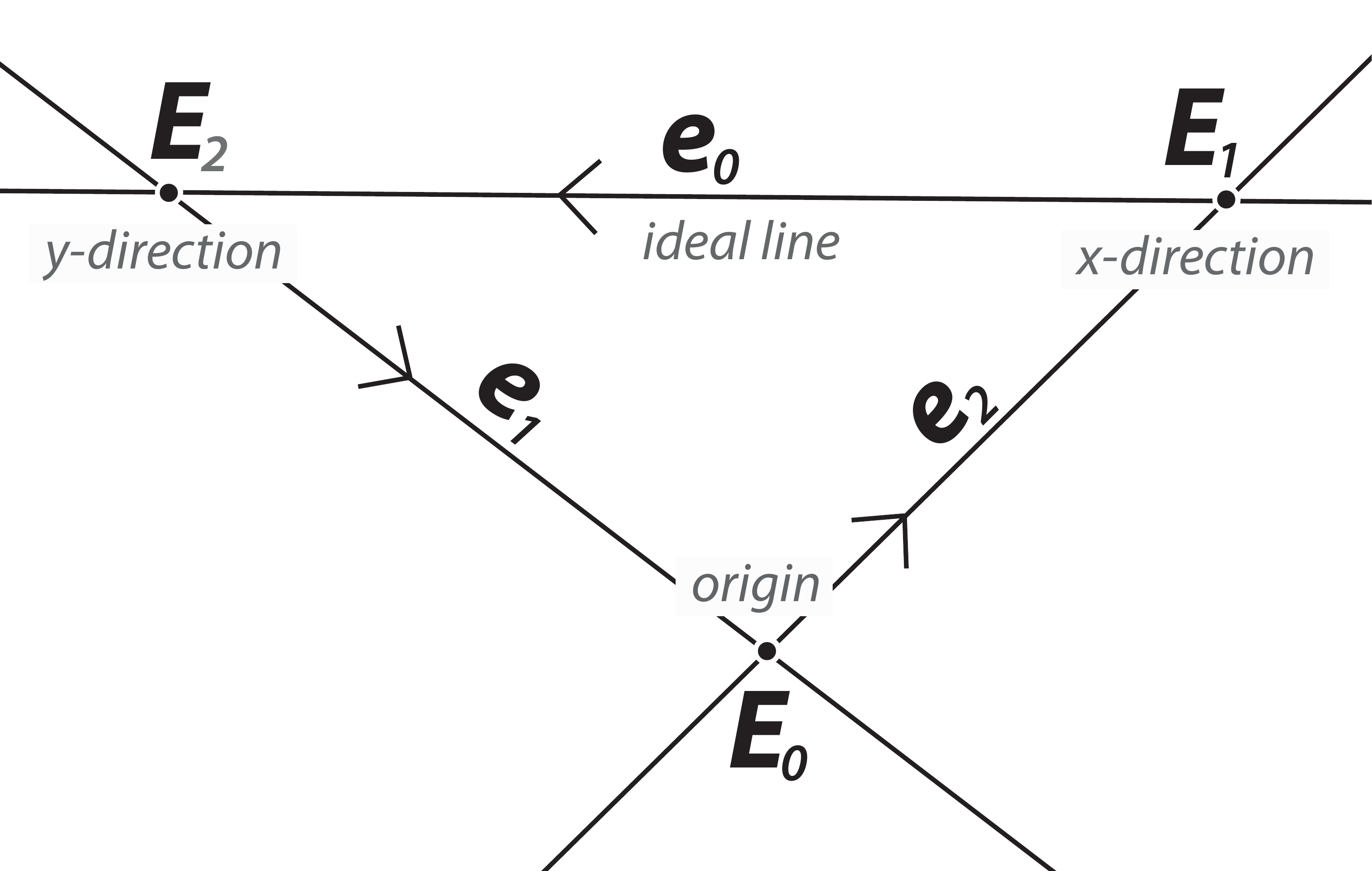}}
  \caption{Perspective view of basis 1- and 2-vectors }
\label{fig:emvl}
\end{figure}

We choose the basis 2-vectors 
\mydisp{\EE{0} := \e{1}\e{2}, ~~\EE{1} := \e{2}\e{0},~~ \EE{2} := \e{0}\e{1}} {} for the points of the plane.  It is easy to check that these satisfy \mydisp{\EE{0}^{2} = -1,~~~~\EE{1}^{2} = \EE{2}^{2} = 0,~~~\EE{i} \cdot \EE{j}= 0~~\text{for}~~i \neq j}{.}   

Hence the induced inner product on points has signature (0,1,2), more degenerate than that for lines.  
\ifthenelse{\equal{\isFV}{true}}{As a result, the distance function between points cannot be obtained from the inner product -- but \emph{can} be obtained via the geometric product;  see  \Sec{sec:pr2pts} below for details.}{}
Points that lie on $\vec{\omega}$ are said to be \emph{ideal}. Then $\EE{0}$ is the origin of the coordinate system, $\EE{1}$ is the {ideal} point in the $x$-direction and $\EE{2}$ is the ideal point in the $y$-direction.   In general, ideal elements can be characterized as elements satisfying $\vec{x}^{2}=0$.    See \Fig{fig:emvl}  for a perspective view of the fundamental triangle determined by these elements.

The basis vectors chosen above assume that the first coordinate is the \emph{homogeneous} coordinate. This assumption is helpful when stating results that should be valid for general dimensions.  On the other hand, existing usage often follows the opposite convention; for example, the line with equation $ax + by + cz = 0$ appears in the algebra as $c\e{0} + a\e{1} + b\e{2}$.  When writing elements of the algebra as tuples, we take into account this existing usage.  We write the 1-vector $\vec{m} = c\e{0} + a\e{1} + b\e{2} $ as $[a,b,c]$ (square brackets), and the 2-vector $\vec{P} = x\EE{1} + y\EE{2} + z\EE{0}$ as $(x,y,z)$ (standard parentheses).

 The pseudoscalar $\eye := \e{0}\e{1}\e{2}$ generates the grade-3 vectors.  It satisfies $\eye^{2} = 0$.  This is, the inner product, or metric, is \emph{degenerate}. 
 A 3-vector $\vec{p}$ has the form $a\eye$ for $a \in \mathbb{R}$. While in a non-degenerate metric the magnitude $a$ of a pseudoscalar $\vec{p}$ can be obtained, up to sign,  as $\vec{p}\eye$, this is not possible with a degenerate metric (since $\vec{p}\eye = a \eye^2 = 0$ for all $\vec{p}$), and we define the signed magnitude $S(\vec{p}) := a$.  We occasionally use the fact for a 1-vector $\vec{a}$ and a 2-vector $\vec{P}$, $S(\vec{a} \wedge \vec{P}) = \vec{a} \vee \vec{P}$.  (This follows from the fact if $\vec{x}\wedge\vec{y}$ is a pseudoscalar, $\vec{x}\vee\vec{y}$ is a scalar with the same magnitude.) Note that  $\EE{i}$ was chosen so that $\e{i}\EE{i}=\eye$. 

\subsection{The geometric product}
The full multiplication table for the basis elements of $\pdclal{2}{0}{1}$  can be found in \Tab{tab:cl201}.  The presence of 0's indicates that the metric is degenerate.
It is useful to have special symbols for the different grade components of the product of two blades, which we now provide.
Let $\vec{A}$ be a $k$-vector and $\vec{B}$, an $m$-vector.  All combinations of $(k,m)$ in $\pdclal{2}{0}{1}$ except $(2,2)$ can then be written as 
\[ \vec{A} \vec{B} = \vec{A}\cdot \vec{B} + \vec{A} \wedge \vec{B} \]
For $(k,m) = (2,2)$,  $\langle \vec{A} \vec{B} \rangle_{2} =: \vec{A} \times \vec{B} ~~(= \vec{A}\vec{B} - \vec{B}\vec{A}$), sometimes called the \emph{commutator} or \emph{cross}  product.  We'll see below that $ \vec{A} \times \vec{B} $ is the ideal point perpendicular to the direction of the joining line of $\vec{A}$ and $\vec{B}$. 

\fvonly{
\begin{table}[b]
\centering
\renewcommand{\arraystretch}{1.1}
\begin{tabularx}{.8\columnwidth} {| Y  || Y | Y  |  Y  | Y | Y | Y | Y | Y  |} \hline 
           & $\one$ & $\e{0}$ & $\e{1}$ & $\e{2}$ & $\EE{0}$ & $\EE{1}$ & $\EE{2}$ & $\eye$  \\ \hline \hline
$\one$        & $\one$ & $\e{0}$ & $\e{1}$ & $\e{2}$ & $\EE{0}$ & $\EE{1}$ & $\EE{2}$ & $\eye$  \\ \hline
$\e{0}$  & $\e{0}$ & $0$     & $\EE{2}$  & $-\EE{1}$ & $\eye$    & $0$      & $0$          & $0      $  \\ \hline
$\e{1}$  & $\e{1}$ & $-\EE{2}$ & $\one$& $\EE{0}$ & $\e{2}$ & $\eye$ & $-\e{0}$ & $\EE{1}$ \\ \hline
$\e{2}$  & $\e{2} $  & $\EE{1}$ & $-\EE{0}$ & $\one$ & $-\e{1}$ & $\e{0}$ & $\eye$ & $\EE{2}$ \\ \hline
$\EE{0}$  & $\EE{0}$ & $\eye$ & $-\e{2}$   & $\e{1}$   & $-\one$ & $-\EE{2}$ & $\EE{1}$ & $-\e{0}$  \\ \hline 
$\EE{1}$  & $\EE{1}$ & $0$     & $\eye$     & $-\e{0}$    & $\EE{2}$ & $0$ & $0$ & $0$  \\ \hline
$\EE{2}$  & $\EE{2}$ & $0$     & $\e{0}$   & $\eye$    & $-\EE{1}$ & $0$ & $0$ & $0$  \\ \hline
$\eye$    & $\eye$     & $0$     & $\EE{1}$ & $\EE{2}$ & $-\e{0}$ & $0$ & $0$ & $0$ \\ \hline
\end{tabularx}
\vspace{.1in}
\caption{Geometric product in $\pdclal{2}{0}{1}$}
\label{tab:cl201}
\end{table}
}


 \subsection{Normalized  points and lines}
 \label{sec:npl}
A $k$-vector whose square is $\pm1$ is said to be \emph{normalized}.  Since normalization simplifies the subsequent discussion, we introduce it here\fvsh{, although logically speaking the justification for all the steps in the normalization process will only later be established.}{.}
 The square of any $k$-vector in the algebra is a scalar, since all $k$-vectors in this algebra are simple. Squaring this product and rearranging terms, one obtains a product of the squares of these 1-vectors, each of which reduces to a scalar. For a euclidean line $\vec{m} = c \e{0} + a\e{1} + b\e{2}$, 
 define the norm \[\| \vec{m} \| := \sqrt{\vec{m}^{2}} = \sqrt{\vec{m}\cdot \vec{m}}~~~ (= \sqrt{a^{2}+b^{2}})\] Then $\vec{m}_{n} := {\| \vec{m} \|^{-1}}{\vec{m}}$ satisfies $\vec{m}_{n}^{2}=1$. For a euclidean point $\vec{P} = z\EE{0}+x\EE{1}+y\EE{2}$, $\vec{P}^{2} = -z^{2}$.  Define $\| \vec{P} \| := z$. Note that, in contrast to a standard norm of a vector space,  $\| \vec{P} \|$ can take on positive \emph{and} negative values, a feature that is occasionally useful. Then $\vec{P}_{n} := {z}^{-1}{\vec{P}}$ satisfies   $\| \vec{P}_{n}\| = 1$.  Such a point is also called \emph{dehomogenized} since its $\EE{0}$ coordinate is 1.  Note that we have shown that normalized euclidean lines have square 1 while normalized euclidean points have square -1.  In the following discussions we often assume that euclidean lines and points are normalized.   
 
 \fvonly{\subsubsection{Weight and norm} 
 \label{secwandf}
 If one has chosen a standard representative $\vec{X}$ for a projective $k$-vector,  and $\vec{Y} = \lambda \vec{X}$, we say that $\vec{Y}$ has \emph{weight} $\lambda$.  We usually choose the standard element to have norm $\pm 1$.  Such elements of weight $\pm 1$ are exactly the normalized elements discussed above.  The weight can be any non-zero real number; while the norm is sometimes restricted to take non-negative values (see \Tab{tab:overview} below). The freedom to choose the weight is a consequence of working in projective space, since non-zero multiples of an element are all projectively equivalent.  Sometimes the weight is irrelevant, sometimes  crucial.  When multiplying elements together, one gets the same projective result regardless of the weights; while adding elements, different weights give different projective results.  }
  
\subsubsection{Ideal elements and free vectors} Ideal points correspond to euclidean \quot{free vectors} (a fact already recognized in \cite{clifford73}). Let  $\vec{P} = a \EE{1} + b\EE{2}$ be an ideal point. Then, as noted above, $\| \vec{P} \| = 0$.  This leads us to  introduce a second norm for ideal points, one that is compatible with their function as free vectors.   Define the \emph{ideal} norm \[\| \vec{P} \|_{\infty} := \| \vec{P} \vee \vec{Q}\|\] where $\vec{Q}$ is \emph{any} normalized euclidean point. Then a direct calculation yields $\| \vec{P} \|_{\infty} = \sqrt{a^{2}+b^{2}}$, as desired.  Thus, the points of the ideal line can be treated as \emph{free vectors} with the  positive definite inner product of $\R{2}$ (signature $(2,0,0)$).   

We write the corresponding inner product between two ideal points $\vec{U}$ and $\vec{V}$ as $\langle \vec{U}, \vec{V} \rangle_{\infty}$.  Every euclidean line $\vec{m}$ has an ideal point $\vec{m}_{\infty}$,  normalized so that $\| \vec{m}_{\infty}\|_{\infty} = 1$.
The ideal norm allows us to represent ideal points  in the accompanying figures as familiar free vectors (arrows labeled with capital letters), see \Fig{fig:sumdiff} (right).  

\begin{table}[t]
\begin{centering}
\renewcommand{\arraystretch}{1.25}
\begin{tabular}{| c | c |  l  |  c | c |}\hline
\textbf{Grade} &  \textbf{Coord. \& tuple form} & \textbf{Norms} & \textbf{Domain} & \textbf{Description} \\ \hline
1 	& {\small $\vec{m} = a \e{1} + b \e{2} + c\e{0} $} & $\|\vec{m}\| := \sqrt{a^2+b^2}$ & $\| \vec{m} \| \neq 0$ & Euc. line \\ \cline{3-5}
 	&  $[a, b, c]$ & $\|\vec{m}\|_\infty := c$ & $\| \vec{m} \| = 0$ & Ideal line \\ \hline
2 	& {\small $\vec{P} = x \EE{1} + y \EE{2} + z\EE{0}$} & $\|\vec{P}\| := z$ & $\| \vec{P} \| \neq 0$ & Euc. point\\ \cline{3-5}
 	&  $(x,y,z)$  & $\|\vec{P}\|_\infty :=\sqrt{x^2+y^2}$ & $\| \vec{P} \| = 0$ & Ideal point  \\ \hline
3 	& $a\eye$  & $\|a\eye\|_\infty :=a$ & -all- & Pseudoscalar \\ \hline
\end{tabular}
\vspace{.1in}
\caption{Coordinate-based overview of the euclidean and ideal elements with their corresponding norms.}
\label{tab:overview}
\end{centering}
\end{table}

We also define an ideal norm for ideal lines (i. e., lines $\vec{m}$ satisfying $\vec{m}^2 = 0$).  For $\vec{m} =  a \e{1} + b \e{2} + c\e{0}$, $\| \vec{m} \|_\infty = c$. (As with $\| \vec{P} \|$ above, this can also take on positive \emph{and} negative values.) Then $\vec{m} = c\omega$.  $c>0$ corresponds to an ideal line in clockwise orientation; $c<0$, to counter-clockwise orientation.  Finally for completeness we can also consider the pseudoscalar signed magnitude $S(a\eye)$ as an ideal norm
: $\| a\eye \|_\infty := S(a\eye) = a$.  We have thus defined an ideal norm for all ideal elements in the algebra.  This ideal norm, restricted to the ideal plane, has signature $(2,0,0)$; considered projectively, this is an elliptic line $\pdclal{2}{0}{0}$, while considered as a vector space, it is $\dclal{2}{0}{0}$,  the geometric algebra of $\R{2}$.  See \Tab{tab:overview} for an overview of the euclidean and ideal elements and norms with their domains of validity.  

In the following, we will  more than once confirm that the standard and ideal norms form an organic whole. For a fuller discussion of the ideal norm see \S 4.4.4 of \gTh.

Whether to apply the standard or ideal inner product  presents no difficulties for practical implementation, as a point can be easily identified as ideal by the linear condition $\vec{P} \wedge \omega = 0$. There is also little danger that an ideal point will be mistaken for a euclidean point -- all the computational paths that produce ideal points presented in this article (see for example \Sec{sec:propseu}, \Sec{sec:pareuc}, \Sec{sec:twoeucpts}, and  \Sec{sec:sumdiffpts}) produce \emph{exact} ideal points. 
This situation is analogous to traditional vector algebra: one has no trouble distinguishing vectors and points. 

  
 \section{The geometric product in detail: 2-way products}
\label{sec:gpdetail}
 
  \fvonly{
    \begin{figure}
   \centering
{\setlength\fboxsep{0pt}\fbox{\includegraphics[width=.7\columnwidth]{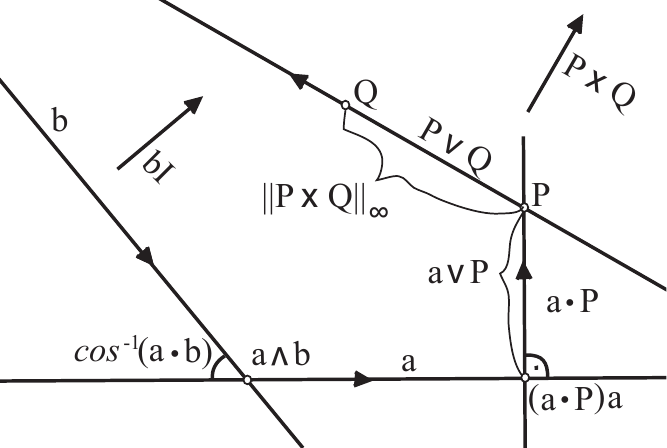}}}
\caption{Selected geometric products of blades.}
\label{fig:geomprod}
\end{figure}
}


In the following discussion,  $\vec{P}$ and $\vec{Q}$ are normalized points (either euclidean or  ideal, as indicated), and $\vec{m}$ and $\vec{n}$ are normalized lines. We analyze  the geometric meaning of products of pairs and triples of $k$-vectors of various grades, paying particular attention to the distinction of euclidean and ideal elements.  A selection of these products is illustrated in \fvsh{\Fig{fig:geomprod}}{\Fig{fig:sumdiff} (left)}. 
  
\subsection{Product with pseudoscalar}
\label{sec:propseu}
First notice that the pseudoscalar $\eye$ commutes with everything in the algebra.  For a  euclidean line $\vec{a}$,  the \emph{polar} point $\vec{a}^{\perp} := \vec{a}\eye = \eye \vec{a}$ is the ideal point perpendicular  to the line $\vec{a}$.  We can use the polar point to define a consistent orientation on euclidean lines; we draw the arrow on an oriented line $\vec{m}$ so that rotating it by $90^{\circ}$ in the CCW direction produces $\vec{m}^{\perp}$.  See \Fig{fig:emvl}, which shows the resulting orientations on the basis 1-vectors.  When $\vec{a}$ is normalized, so is $\vec{a}^{\perp}$, another confirmation that the two norms (euclidean and ideal) have been harmoniously chosen.  For a normalized  euclidean point $\vec{P}$,   $~~\vec{P}^{\perp} := \vec{P}\eye = \eye \vec{P} = -\e{0}$, the ideal line with CW orientation.  The polar of an ideal point or line is 0.  

 \fvonly{We noted above in \Sec{sec:bva} that the condition $\eye^{2}=0$ means the metric is degenerate, or, what is the same, multiplication by $\eye$ (the so-called metric polarity) is not an algebra isomorphism.  Although some researchers see this as a flaw in the algebra (for example, \cite{li08},  p. 11), our experience leads to view it as an advantage, since it accurately mirrors the metric relationships in the euclidean plane.  For example, when $\vec{m}$ and $\vec{n}$ are parallel, $\vec{m}^{\perp} = \vec{n}^{\perp}$, that is, parallel lines have the same polar point.  In a non-degenerate metric, however, different lines have different polar points.  In contrast, the degenerate metric properly mirrors this euclidean phenomenon. For a fuller discussion of this theme, see Sec. 5.3 of \cite{gunn2016}.}

\subsection{Product of two lines}
\label{sec:pr2lns}
  
   \begin{figure}
   \centering
{\setlength\fboxsep{0pt}\fbox{\includegraphics[width=.47\textwidth]{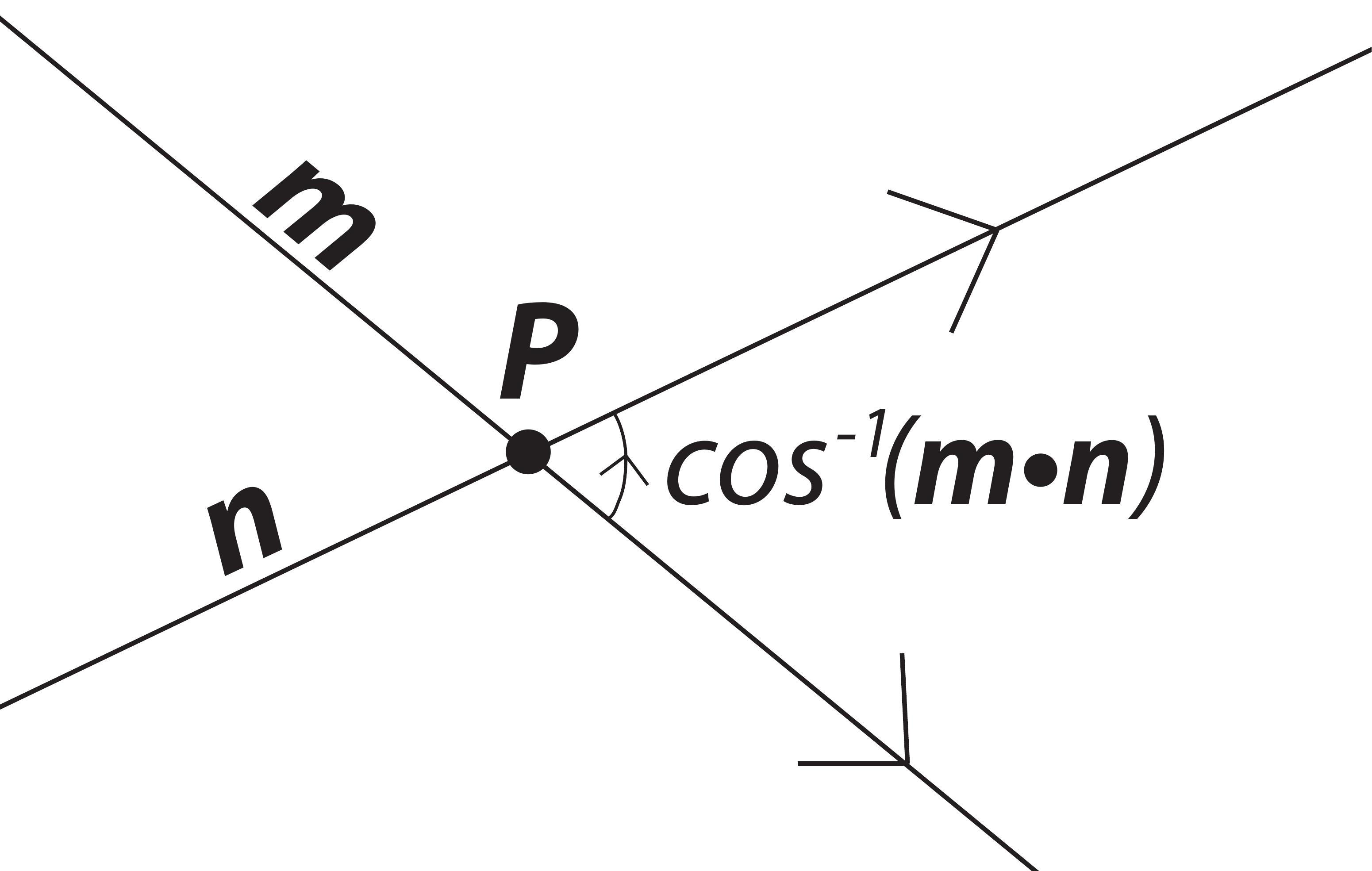}}} \hspace{.02\textwidth}
{\setlength\fboxsep{0pt}\fbox{\includegraphics[width=.47\textwidth]{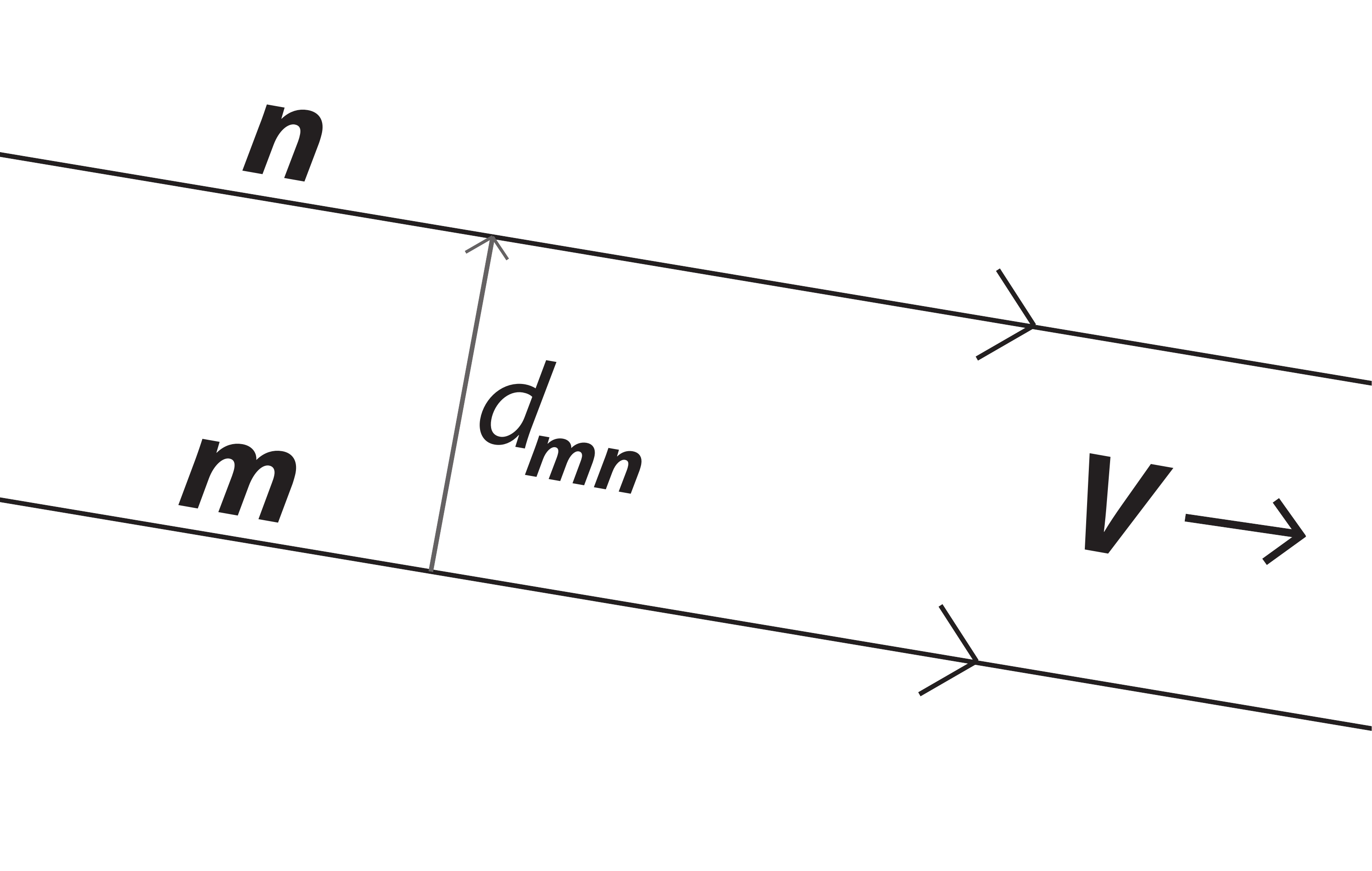}}}\caption{Geometric product $\vec{a}\vec{b}$  of two intersecting lines (left) and two parallel lines (right).}
\label{fig:twoLines}
\end{figure}

In general we have $\vec{m}\vec{n} = \langle \vec{m}\vec{n}\rangle_{0} +\langle \vec{m}\vec{n}\rangle_{2} =  \vec{m} \cdot \vec{n} + \vec{m} \wedge \vec{n}$.  We say  two lines are \emph{perpendicular} if $\vec{m} \cdot \vec{n} = 0$ -- even when one of the lines is ideal. The meaning of the two terms on the right-hand side depends on the configuration of $\vec{m}$ and $\vec{n}$ as follows. 

\subsubsection{Intersecting euclidean lines} We say that two intersecting euclidean lines meet at an angle $\alpha$ when a rotation of $\alpha$ around their common point brings the first oriented line onto the second, respecting the orientation.  
Then $\vec{m} \cdot \vec{n} = \cos{\alpha}$ and $\vec{m} \wedge \vec{n} = (\sin{\alpha}) \vec{P}$ where $\vec{P}$ is their normalized intersection point. Consult \Fig{fig:twoLines}, left. Readers who are surprised that the angle $\alpha$ can be deduced from the wedge product -- which doesn't depend on the metric -- are reminded that this is possible only because we have used the inner product to normalize the arguments in advance.   \fvonly{Without normalizing $\vec{m}$ and $\vec{n}$, the formulae are \[ \vec{m} \cdot \vec{n} = \|\vec{m}\| \|\vec{n}\| \cos{\alpha}~~~~\text{~~~~and~~~~}~~~~\vec{m} \wedge \vec{n} = \|\vec{m}\| \|\vec{n}\| (\sin{\alpha}) \vec{P}\] Similar extensions involving non-normalized arguments could be made for the subsequent formulae given below, but in the interests of space we omit them.} 

\myexercise$(\vec{m}\vec{n})^{n} = \cos{n \alpha} + (\sin{n \alpha}) \vec{P}$.  \fvonly{Show that the vector subspace generated by $1$ and $\vec{P}$ is isomorphic to the complex plane $\mathbb{C}$.} 

\subsubsection{Parallel euclidean lines} 
\label{sec:pareuc}
$\vec{m} \cdot \vec{n} = \pm1$.  We say the lines are \emph{parallel} when this inner product equals $1$, otherwise we say they are \emph{anti-parallel}.  In the latter case, replace $\vec{n}$ by $-\vec{n}$ to obtain parallel lines.  Then $\vec{m} \cdot \vec{n} = 1$ and $\vec{m} \wedge \vec{n} = d_{\vec{m}\vec{n}} \vec{m}_{\infty}$, where $d_{\vec{m}\vec{n}}$ is the oriented euclidean distance between the two  lines and $\vec{m}_\infty$.  See \Fig{fig:twoLines}, right. The simplicity of this formula validates the choice of the norm $\| \|_{\infty}$ on ideal points.   Note that the geometric product in PGA automatically finds the correct form of measuring the \quot{distance} between the two lines: the weight of the intersection point $\vec{m}\wedge\vec{n}$ reflects angle measurement $(\sin{\alpha})$ for intersecting lines and euclidean distance measurement ($d_{\vec{m}\vec{n}}$) for two parallel lines. 

\myexercise  $(\vec{m}\vec{n})^{n} = 1 + {n d_{\vec{m}\vec{n}} \vec{m}_{\infty}}$. 

\subsubsection{Product of a euclidean line with the ideal line} Let $\vec{n} =\vec{\omega}$ be the ideal line. Then $\vec{m} \cdot \vec{n} = 0$ and $\vec{m} \wedge \vec{n} = \vec{m}_{\infty}$ is the ideal point of $\vec{m}$.  Note that since $\vec{m} \cdot \vec{n} = 0$, the ideal line is perpendicular to every euclidean line; since it shares an ideal point with each such line, it is  parallel to every euclidean line!

\subsection{Product of two points}
\label{sec:pr2pts}
Here the general formula is $\vec{P} \vec{Q} =  \langle \vec{P}\vec{Q}\rangle_{0} +\langle \vec{P}\vec{Q}\rangle_{2} = \vec{P} \cdot \vec{Q} + \vec{P} \times \vec{Q}$.  \ifthenelse{\equal{\isFV}{true}}{The resulting behavior is characterized by the fact that the inner product for points is more degenerate than that for lines. }{}

   \begin{figure}
   \centering
{\setlength\fboxsep{0pt}\fbox{\includegraphics[width=.47\textwidth]{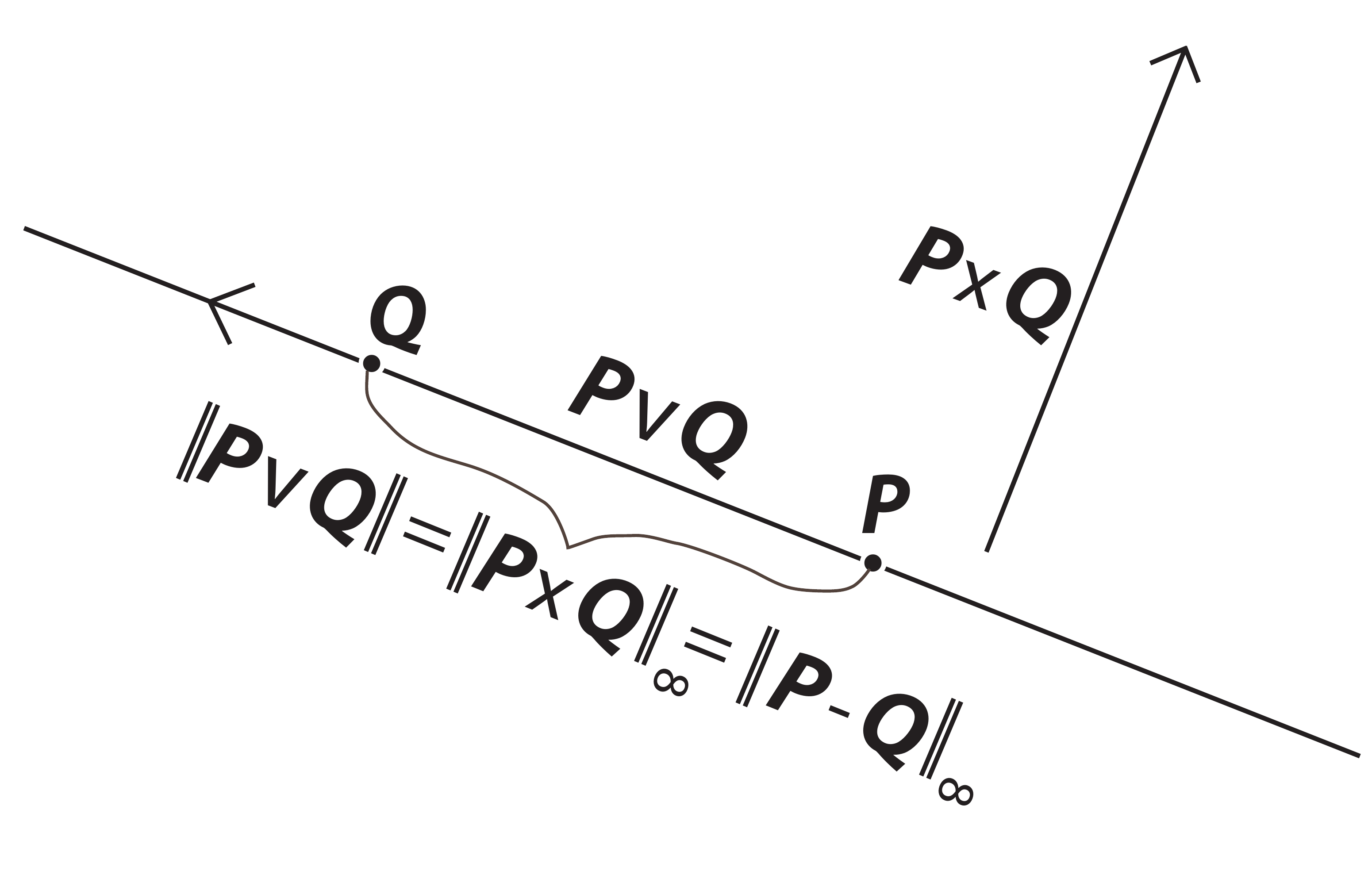}}} \hspace{.02\textwidth}
{\setlength\fboxsep{0pt}\fbox{\includegraphics[width=.47\textwidth]{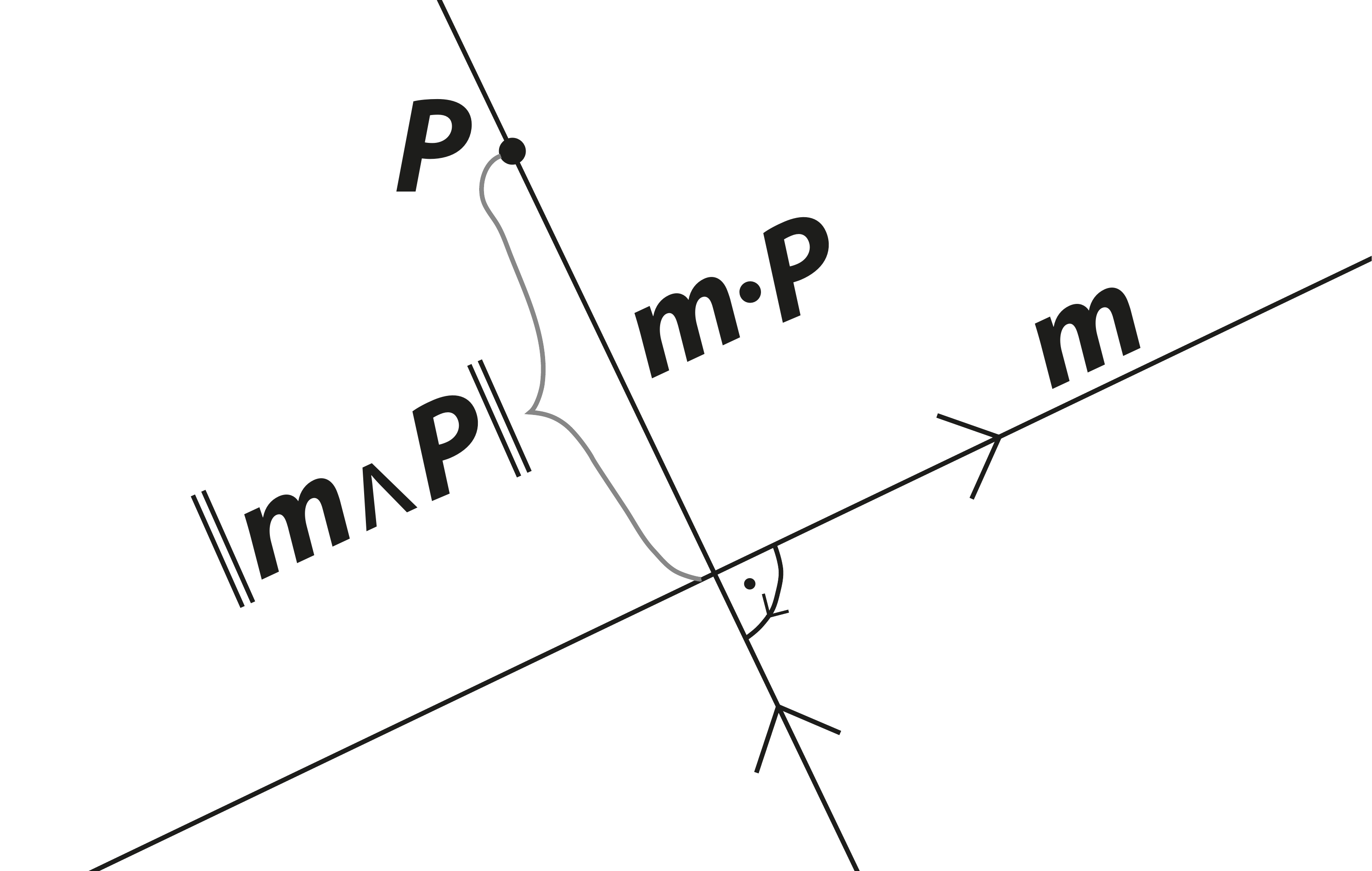}}}\caption{\textit{Left:}  product $\vec{P}\vec{Q}$ of two euclidean points; \textit{Right:}  product $\vec{a}\vec{P}$ of euclidean line and point.}
\label{fig:twoPoints}
\end{figure}

\subsubsection{Two euclidean points}   
\label{sec:twoeucpts}
$\vec{P} \cdot \vec{Q} = -1$ and $\vec{P} \times \vec{Q}$ is an ideal point perpendicular to $\vec{P} \vee \vec{Q}$. To be exact  $\vec{P} \times \vec{Q} = -(\vec{P} \vee \vec{Q})\eye$ (notice the negative sign). We also write this as $(\vec{P}-\vec{Q})^{\perp}$ since the ideal point $\vec{P} - \vec{Q}$, rotated in the CCW direction by $90^{\circ}$, yields $\vec{P}\times \vec{Q}$.   See \Fig{fig:twoPoints}, left. 
 
 \myexercise  The distance  $d_{\vec{P}\vec{Q}}$ between two euclidean points 
 satisfies \[d_{\vec{P}\vec{Q}} = \| \vec{P} \times \vec{Q} \| ~~(= \| \vec{P} \vee \vec{Q} \|)\]

\subsubsection{Euclidean point and ideal point}  If $\vec{Q}$ is ideal, then $\vec{P} \cdot \vec{Q} = 0$ and $\vec{P} \times \vec{Q}$ is the ideal point obtained by rotating $\vec{Q}$ $90^{\circ}$ in the CW direction.  This result is consistent with the characterization of the product of two euclidean points: it is an ideal point perpendicular to $\vec{P} \vee \vec{Q}$. $\vec{Q} \times \vec{P}$ rotates in the CCW direction. Thus, multiplication of an ideal point by any finite point rotates the ideal point by $90^{\circ}$; the specific location of the euclidean point plays no role.  
  
\subsubsection{Two ideal points} The product of two ideal points is zero.  Hence the only interesting binary operation on ideal points is addition.  In light of \Sec{secwandf}, this helps to explain why ideal points are often treated as vectors rather than projective points.


\subsection{Product of a line and a point}
\label{sec:prlnpt}
The general formula is $\vec{m} \vec{P} =  \langle \vec{m}\vec{P}\rangle_{1} +\langle \vec{m}\vec{P}\rangle_{3} = \vec{m} \cdot \vec{P} + \vec{m} \wedge \vec{P}$.  The wedge vanishes if and only if $\vec{P}$ and $\vec{m}$ are incident.  As before, we assume that both $\vec{m}$ and $\vec{P}$ are normalized.

\subsubsection{Euclidean line and euclidean point}   $ \vec{m} \cdot \vec{P} $ is the line passing through $\vec{P}$ perpendicular to $\vec{m}$ (consult \Fig{fig:twoPoints}, right). Why? This can be visualized as starting with all the lines through $\vec{P}$ and removing all traces of the line parallel to $\vec{m}$, leaving the line perpendicular to $\vec{m}$.  
It has the same norm as $\vec{m}$, and its orientation is obtained from that of $\vec{m}$ by CCW rotation of $90^{\circ}$.  This is reversed in the product $\vec{P} \cdot \vec{m}$.  This sub-product is important enough to deserve its own symbol.  We define \[\vec{m}^{\perp}_{\vec{P}} := \vec{m} \cdot \vec{P} = - \vec{P} \cdot \vec{m}\]  The wedge product satisfies $\vec{m}\wedge \vec{P} = d_{\vec{m}\vec{P}}\eye$, where $d_{\vec{m}\vec{P}}$ is the directed distance between $\vec{m}$ and $\vec{P}$.   

\subsubsection{Euclidean line and ideal point}  Let $\alpha$ be the angle between the direction of $\vec{m}$ and $\vec{P}$: $\cos{\alpha} = \langle \vec{m}_{\infty}, \vec{P} \rangle_{\infty}$. Then $ \vec{m} \cdot \vec{P} = (\cos{\alpha})\vec{\omega}$ and $\vec{m}\wedge \vec{P} = (\sin{\alpha}) \eye$.  
Notice that $\vec{m} \vec{P}$ is the sum of an  ideal line and a pseudoscalar: no euclidean point or line appears in the product. The first term, involving the ideal line, is non-zero when the ideal line is the only line through $\vec{P}$ perpendicular to $\vec{m}$.  When $\alpha = \frac{\pi}{2}$, every line through $\vec{P}$ is   perpendicular to $\vec{m}$, and $\vec{m} \cdot \vec{P} = 0$ while $\vec{m}\wedge \vec{P}  = \eye$.

 \section{The geometric product in detail: 3-way products}
\label{sec:threes}
Products of more than 2 $k$-vectors can be understood by multiplying the factors out, one pair at a time.  The product of 3 different euclidean points (or lines) is important enough in its own right to merit a separate discussion. The results provide a promising basis for  a future investigation of euclidean triangles.  Later we will see that euclidean reflections (\Sec{sec:refl}) and orthographic projection (\Sec{sec:orthpro}) can also be understood as 3-way products in which one of the factors is repeated.

  \begin{figure}
   \centering
{\setlength\fboxsep{0pt}\fbox{\includegraphics[width=.7\columnwidth]{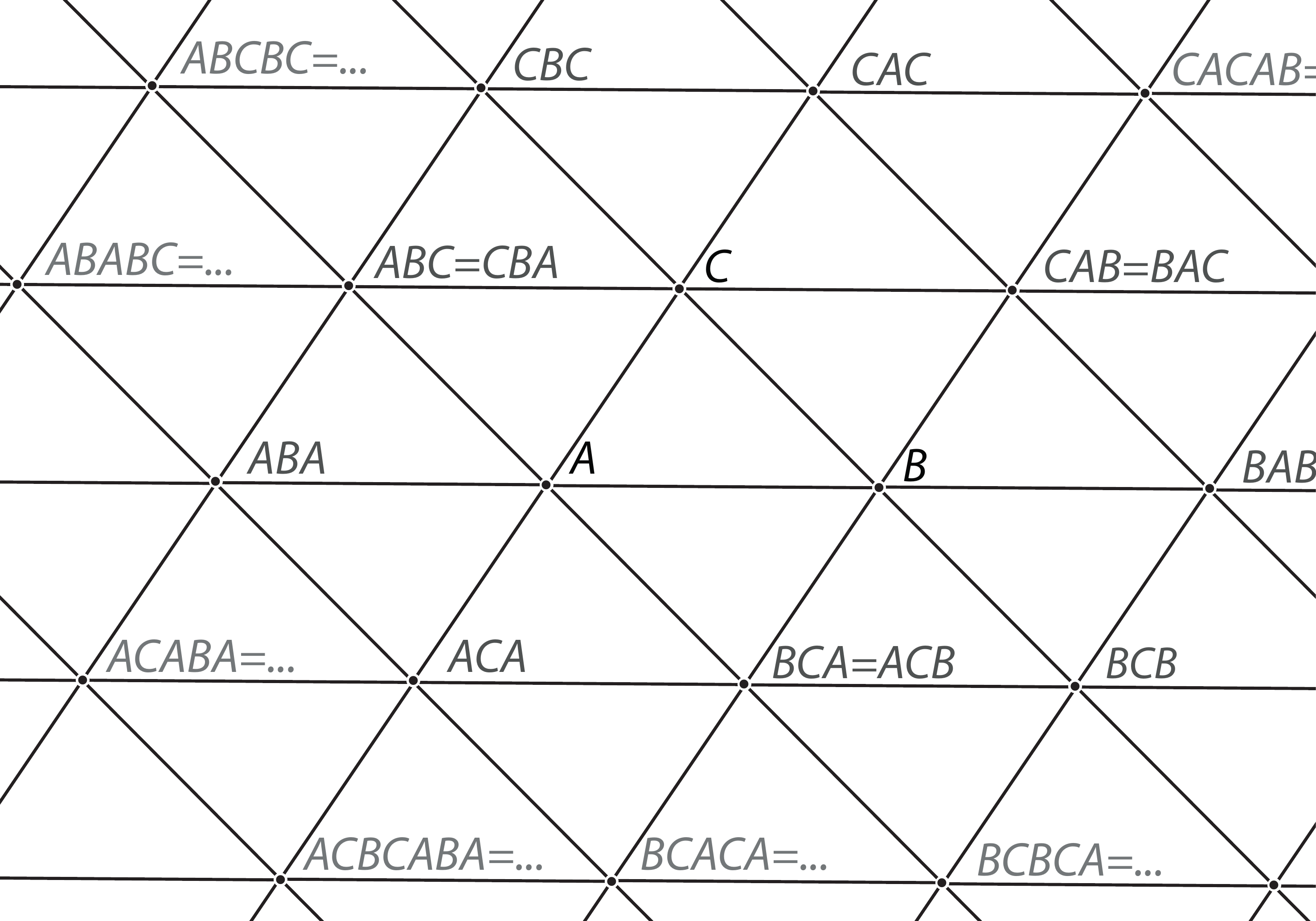}}}\hspace{.03in}
\caption{Products of 3 euclidean points}
\label{fig:trianglepts}
\end{figure}

 \subsection{Product of 3 euclidean points}
 Let the 3 points be $\vec{A}$, $\vec{B}$, and $\vec{C}$. See \Fig{fig:trianglepts}. Then using the results obtained above for products of two points:
 \begin{align*}
 \vec{A}\vec{B}\vec{C} &= (\vec{A} \vec{B})\vec{C} \\
 &= (-1 + (\vec{A} - \vec{B})^{\perp}) \vec{C} \\
 &= -\vec{C} - (\vec{A}-\vec{B}) \\
 &= \vec{A} -\vec{B} +\vec{C}
 \end{align*}
 The first and second steps follow from the results from \Sec{sec:pr2pts}.  The final equation indicates the projective equivalence of the two expressions, since multiplying by $-1$ does not effect the projective point.  The result is somewhat surprising, since the scalar part vanishes.  
 Hence, if one begins with the triangle $\vec{A}\vec{B}\vec{C}$ and generates a lattice of congruent triangles by translating the triangle along its sides, then the vertices of this lattice can be labeled by products of odd numbers of the vertices $\vec{A}$, $\vec{B}$, and $\vec{C}$ (\Fig{fig:trianglepts}).

 \myexercise  
 The product of an odd number of euclidean points is  a euclidean point that is the alternating sum of the arguments. 
 
  \begin{figure}[b]
   \centering
{\setlength\fboxsep{0pt}\fbox{\includegraphics[width=.7\columnwidth]{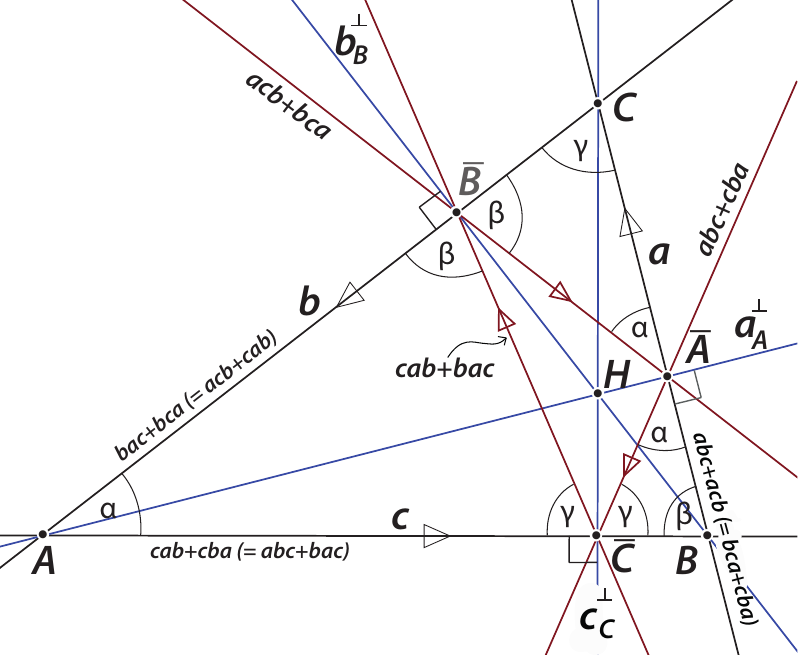}}}
\caption{Product of 3 euclidean lines.}
\label{fig:trianglelns}
\end{figure}

\subsection{Product of 3 euclidean lines}
 \label{sec:threelns} 
Let the 3 (normalized) lines be $\vec{a}$, $\vec{b}$, and $\vec{c}$ oriented cyclically.  See \Fig{fig:trianglelns}.  These three lines determine a triangle.  Then $\vec{a} \wedge \vec{b} = \sin{(\pi - \gamma)}\vec{C}$, etc., produces the interior angle $\gamma$ and the (normalized) vertex $\vec{C}$ of the triangle.  Using the results obtained above for products of two lines:
 \begin{align*}
 \vec{a}\vec{b}\vec{c} &= (\vec{a} \vec{b})\vec{c} \\
&= ( -\cos{\gamma})\vec{c} + (\sin{\gamma})(\vec{C} \vec{c})  \\
 &=  (-\cos{\gamma})\vec{c} + (\sin{\gamma})(\vec{C} \cdot \vec{c} + \vec{C}\wedge\vec{c}) \\
 &=  -((\cos{\gamma})\vec{c} + (\sin{\gamma})\vec{c}^{\perp}_{\vec{C}}) + \sin{\gamma}d_{\vec{C}\vec{c}}\eye 
 \end{align*}
  The first  step follows from the results from \Sec{sec:pr2lns}, the second and third from \Sec{sec:prlnpt}.   Let $\overline{\vec{C}}$ be the intersection of $\vec{c}$ and $\vec{C}\cdot\vec{c}$. In the last equation the expression in parentheses is the grade-1 part of the product: $\overline{\vec{b}} := \langle  \vec{a}\vec{b}\vec{c}  \rangle_{1}$. It is, by inspection, minus the result of rotating $\vec{c}$ around $\overline{\vec{C}}$ by $\gamma$.
Parenthesizing in a different order yields:
 \begin{align*}
 \vec{a}\vec{b}\vec{c} =\vec{a}(\vec{b}\vec{c}) &=  -((\cos{\alpha})\vec{a} - (\sin{\alpha})\vec{a}^{\perp}_{\vec{A}}) + \sin{\alpha}d_{\vec{A}\vec{a}}\eye
 \end{align*}
In this form, $\overline{\vec{b}} $ is minus the result of rotating $\vec{a}$ around $\overline{\vec{A}}$ by $-\alpha$.  Hence $\overline{\vec{b}} $ must be the joining line of $\overline{\vec{A}}$ and $\overline{\vec{C}}$. See \Fig{fig:trianglelns}.
\fvsh{

Since the grade-3 parts are equal, one obtains: \[(\sin{\gamma})d_{\vec{C}\vec{c}} = (\sin{\alpha})d_{\vec{A}\vec{a}}\]
This  illustrates an important technique for generating formulas in geometric algebra.  By applying the associative principle one can insert parentheses at different positions: \[ (\vec{a} \vec{b})\vec{c}= \vec{a}\vec{b}\vec{c} = \vec{a} ( \vec{b}\vec{c}) \]  The left-hand side and right-hand side represent different paths in the algebra to the same result, and these often produce non-trivial identities as this one. 
}{Since the grade-3 parts are equal, one also obtains: $(\sin{\gamma})d_{\vec{C}\vec{c}} = (\sin{\alpha})d_{\vec{A}\vec{a}}$.   Associativity of the geometric product yields many formulas of this sort. We meet this technique again in the discussion of orthographic projection in \Sec{sec:orthprod} below.}  

\myboldhead{Exercises} 1)  $\langle  \vec{a}\vec{b}\vec{c}  \rangle_{1}$ = $\frac{1}{2}(\vec{a}\vec{b}\vec{c} + \vec{c}\vec{b}\vec{a})$. 2) $\frac{1}{2}(\vec{c}\vec{a}\vec{b} + \vec{c}\vec{b}\vec{a}) = \cos(\gamma)\vec{c}  $. 3) Define \[\vec{s} := \vec{abc} + \vec{acb} + \vec{bac} + \vec{bca} + \vec{cab} + \vec{cba}\]  Show that $\vec{s}$ is a 1-vector, called the \emph{symmetric line} of the triple $\{\vec{a},\vec{b},\vec{c}\}$.  
 \Fig{fig:trianglelns} illustrates these relations, and illustrates how the geometric product in PGA produces compact and elegant expressions for familiar triangle constructions.  

\section{Distance and angle formulae}
\label{sec:daf}
We collect here the various distance formulae  encountered in the process of discussing the 2-way vector products above.  $\vec{P}$ and $\vec{Q}$ are normalized euclidean points, $\vec{U}$ and $\vec{V}$ are normalized ideal points, and $\vec{m}$ and $\vec{n}$ are normalized euclidean lines.  Space limitations prevent further differentiation with respect to signed versus unsigned distances. Consult \Fig{fig:geomprod}.
\begin{compactenum}
\item \textbf{Intersecting lines.} $\angle(\vec{m},\vec{n}) = \cos^{-1}{(\vec{m} \cdot \vec{n})} = \sin^{-1}{(\| \vec{m} \wedge \vec{n} \|)}$
\item \textbf{Parallel lines.} $d(\vec{m}, \vec{n}) = \|\vec{m} \wedge \vec{n}\|_\infty$
\item \textbf{Euclidean points.} $d(\vec{P},\vec{Q}) = \|\vec{P} \vee \vec{Q} \| = \| \vec{P} \times \vec{Q} \|_\infty$
\item \textbf{Ideal points.} $\angle(\vec{U}, \vec{V}) = \cos^{-1}(\langle \vec{U}, \vec{V} \rangle_{\infty})$
\item \textbf{Euclidean line, euclidean point.} $d(\vec{m}, \vec{P})  = -d(\vec{P}, \vec{m}) = S(\vec{m} \wedge \vec{P}) = \vec{m} \vee \vec{P} $
\item \textbf{Euclidean line, ideal point.} $\angle(\vec{m}, \vec{U}) = \cos^{-1}{(\| \vec{m} \cdot \vec{U} \|_{\infty})}$
\end{compactenum}
Notice that  a single expression in the geometric algebra produces several correct variants which take into account whether one or the other or both of the arguments are ideal.  For example, $ \|\vec{m} \wedge \vec{n}\|$ produces the intersection point of the two lines weighted by either the inverse of the sine of the angle (when the lines intersect), or the euclidean distance between them (when they are parallel).  Similar phenomena reveal themselves also in the next section.
\vspace{-.1in}
 
 \section{Sums and differences of points and of lines}
 \label{sec:sumdiff}
 Based on the discussion of the geometric product above,
it is instructive to examine sums and differences of points, resp. lines.  This deceptively simple theme reveals important distinctions between euclidean and ideal points and lines that play a central role throughout this algebra. It also highlights how traditional vector algebra can be directly accessed within $\pdclal{2}{0}{1}$ (as the weighted ideal points).   As before, all points and lines are assumed to be normalized unless otherwise stated. Consult \Fig{fig:sumdiff}.
 
  \begin{figure}[b]
   \centering
{\setlength\fboxsep{0pt}\fbox{\includegraphics[width=.4\columnwidth]{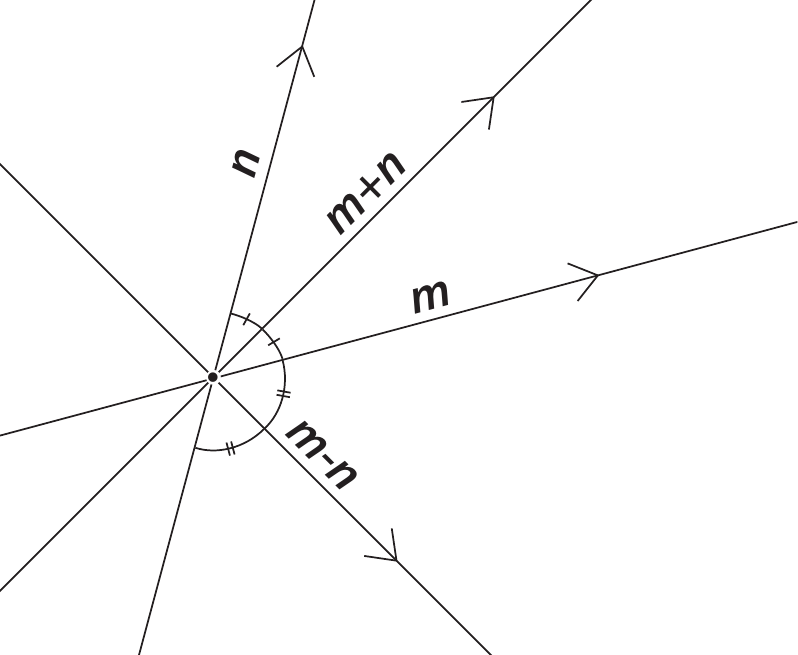}}\hspace{.1in}
{\setlength\fboxsep{0pt}\fbox{\includegraphics[width=.4\columnwidth]{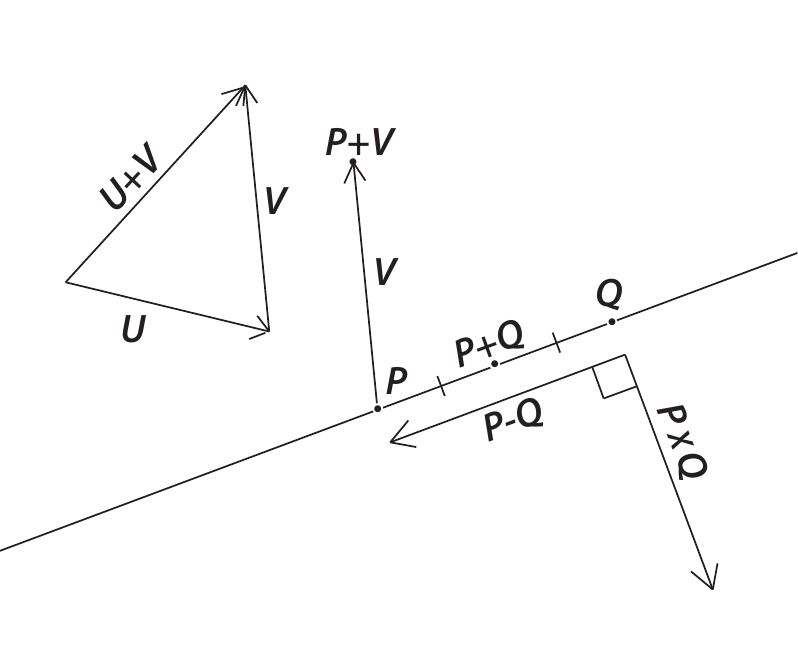}}}}
\caption{\emph{Left}: Sums and differences of normalized euclidean lines.  \emph{Right}: Sums and differences involving ideal points and normalized euclidean points. $\vec{P}  + \vec{Q}$ is the (non-normalized) midpoint of segment $\overline{\vec{P}\vec{Q}}$; $\vec{P}\times\vec{Q}$ is the ideal point $\vec{P}-\vec{Q}$ rotated $90^\circ$ CCW.}
\label{fig:sumdiff}
\end{figure}

\subsection{Sums and differences of lines} When $\vec{m}$ and $\vec{n}$ are both  euclidean,  and intersect in a euclidean point, then $\vec{m} + \vec{n}$ is their mid-line, the line through their common point $\vec{m} \wedge \vec{n}$ that bisects the angle between $\vec{m}$ and $\vec{n}$.  $\vec{m} - \vec{n}$  also passes through their common point, but bisects the supplementary angle between the two lines. (To establish the claim, consider the inner product of $\vec{m} \pm \vec{n}$ with each line separately.) If the two lines are parallel, then  $\vec{m} + \vec{n}$ is their mid-line: the line parallel to both, halfway in between them.    $\vec{m} - \vec{n}$ is the ideal line, weighted by the signed distance between the lines. If $\vec{m}$ is euclidean and $\vec{n} = \lambda \vec{\omega}$ is a weighted ideal line, then $\vec{m} + \vec{n}$ is a (normalized) euclidean  line representing the translation of the line $\vec{m}$ by a signed distance $\lambda$ in the direction perpendicular to its own direction (to  be exact, in the direction opposite its polar point $\vec{m}^{\perp}$).  
 
 \subsection{Sums and differences of points} 
 \label{sec:sumdiffpts}
 When $\vec{P}$ and $\vec{Q}$ are both euclidean, $\vec{P} + \vec{Q}$ is their mid-point.  ($\frac{\vec{P} + \vec{Q}}{2}$ is the normalized mid-point.)  $\vec{P} - \vec{Q}$ is an ideal point representing their vector difference.  If $\vec{P}$ is normalized euclidean and $\vec{V}$ is ideal (not necessarily normalized), then $\vec{P} \pm \vec{V}$ is a (normalized) euclidean  point representing the translation of the point $\vec{P}$ by the free vector $\pm\vec{V}$. If both $\vec{U}$ and $\vec{V}$ are ideal (again, not necessarily normalized), then $\vec{U} \pm \vec{V}$ is the ideal point representing their vector sum (difference).  Here we once again meet the $\R{2}$ vector space structure on the ideal line induced by the ideal norm. 

%

\vspace{-.1in}

\section{Isometries}
\label{sec:isom}
\fvsh
{
Equipped with our detailed knowledge of 2-way products we now turn to discuss how to implement euclidean isometries in the algebra. 
Recall that the group of isometries of $\Euc{2}$ is generated by reflections in euclidean lines.  The product of an even number of reflections yields a direct (orientation-preserving) isometry (either a rotation or a translation), while an odd number produces an indirect (orientation-reversing) isometry. 
Also recall, that in the euclidean plane, every isometry can be written using 1, 2, or 3 reflections. 
We now show how to implement reflections using the geometric product, then extend this result to products of 2 and 3 reflections.

\subsection{Reflections}
\label{sec:refl}
Suppose $\vec{a}$ and $\vec{b}$ are two normalized euclidean lines, and let $R_{\vec{a}}(\vec{b})$ represent the reflection of $\vec{b}$ in $\vec{a}$.  Purely geometric considerations imply that $R_{\vec{a}}(\vec{b})$ is a line $\vec{x}$ satisfying $\vec{a} \cdot \vec{x} = \vec{a} \cdot \vec{b}$ and $\vec{a} \wedge \vec{x} = \vec{b} \wedge \vec{a}$. 

 \myexercise  Show that $\vec{x} := \vec{a}\vec{b}\vec{a}$ fulfils both conditions, satisfies $\vec{x} \neq \vec{b}$ when $\vec{a} \neq \vec{b}$ and hence is the desired reflection. 
  \begin{figure}
   \centering
{\setlength\fboxsep{0pt}\fbox{\includegraphics[width=.6\columnwidth]{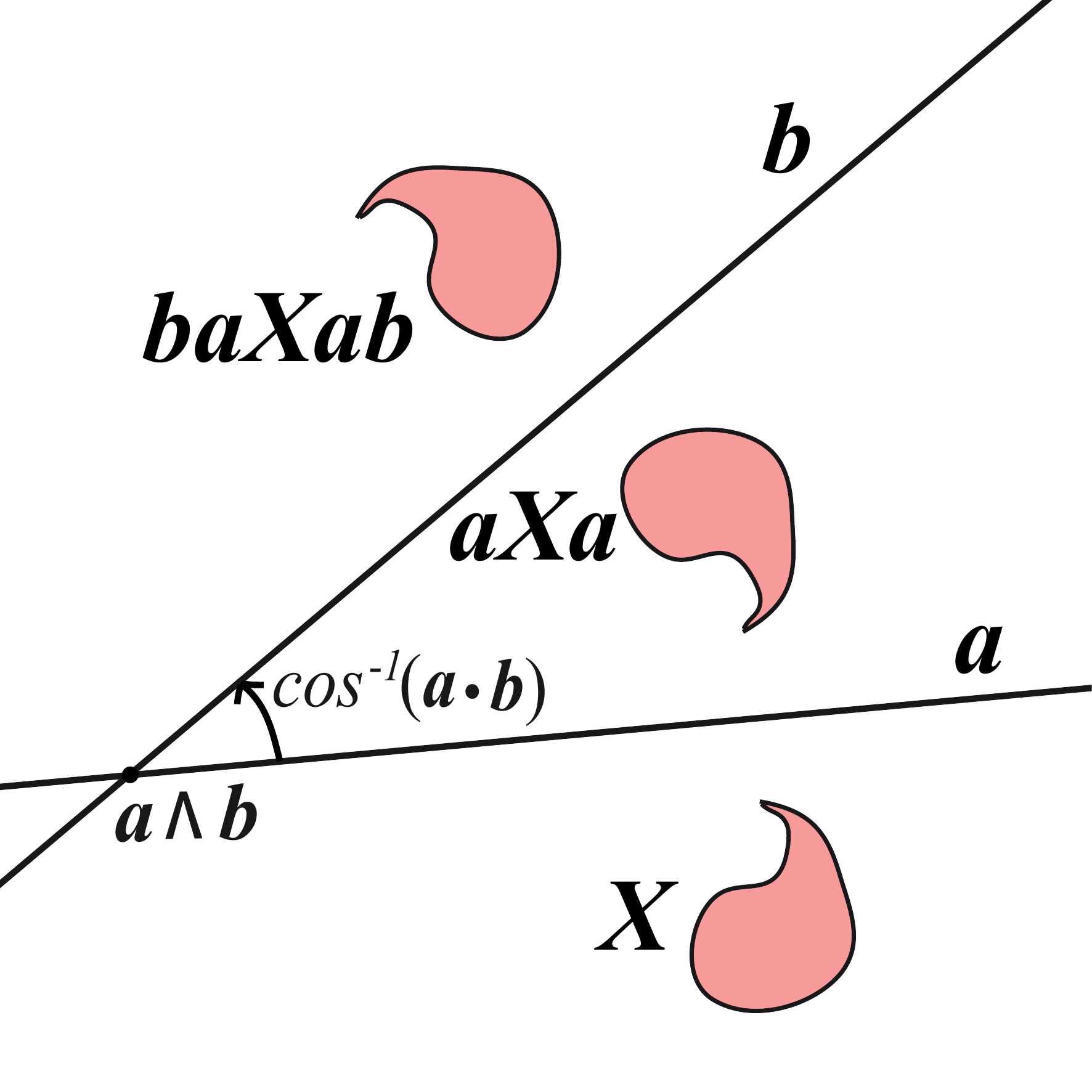}}}
\caption{The reflection in the line $\vec{a}$ is implemented by the sandwich $\vec{a}\vec{X}\vec{a}$; the product of the reflection in line $\vec{a}$ followed by reflection in (non-parallel) line $\vec{b}$ is a rotation around their common point $\vec{a}\wedge\vec{b}$ through $2 \cos^{-1}(\vec{a}\cdot\vec{b})$.}
\label{fig:rotationab}
\end{figure}

Notice that a reflection can then be seen as a special form of a 3-way product in which the first and third term is the same line.
 We write the reflection operator $\vec{b} \rightarrow \vec{a}\vec{b} \vec{a}$ as $\overline{\vec{a}}(\vec{b})$.   We sometimes refer to this as a \emph{sandwich} operator since the $\vec{a}$ \quot{sandwiches} the operand $\vec{b}$ on both sides. 

\myexercise  Show that $\overline{\vec{a}}(\vec{P})$ is also a reflection applied to a euclidean point $\vec{P}$. [Hint: Write $\vec{P}= \vec{m}\vec{n}$ for orthogonal $\vec{m}$ and $\vec{n}$.] 

\subsection{Product of two reflections}
\label{sec:prtworef}
Before we discuss the product of several  reflections,  we introduce some terminology. The product of any number of euclidean lines is called a \emph{versor};  the product of an even number is called a \emph{rotor}.  Versors and rotors are important since sandwich operators based on them yield euclidean isometries.

The concatenation of two reflections  in lines $\vec{a}$ and $\vec{b}$ can be written \[\overline{\vec{b}}(\overline{\vec{a}}(\vec{x})) = \vec{b}(\vec{a}\vec{x}\vec{a})\vec{b} = (\vec{b}\vec{a})\vec{x}(\vec{a}\vec{b})\]  where the expression on the right is obtained by applying associativity to the middle expression.  Define  $\vec{r} := \vec{b}\vec{a}$, and an operator $\overline{\vec{r}}(\vec{x}):= \vec{r} \vec{x}\widetilde{\vec{r}}$ which represents the composition of these two reflections expressed using the rotor  $\vec{r}$. Such a composition can take two forms, depending on the position of the lines.  

When the lines intersect in a euclidean point, then $\overline{\vec{r}}$  is a rotation around that point by twice the angle between the lines.  See \Fig{fig:rotationab}.  When the lines are a parallel, $\overline{\vec{r}}$ is a translation by twice the distance between the lines in the direction perpendicular to the direction of the lines.  The details can be confirmed by applying the results above involving products of two lines  in \Sec{sec:pr2lns} to write out $\vec{r}$ for these two cases and then by multiplying out the resulting sandwich operators.  The rotor for a rotation is called a \emph{rotator}; for a translation, a \emph{translator}.

\myexercises 1) Show that for a translator $\vec{t}$, $\vec{t}\vec{x} = \vec{x}\widetilde{\vec{t}}$ represents half the translation of the sandwich $\overline{\vec{t}}(\vec{x})$. That is, translators also make good \quot{open-faced} sandwiches. 2) Discuss the rotator $\cos{\alpha} + (\sin{\alpha})\vec{P}$ when $\alpha = \frac{\pi}{2}$. 

\subsection{Product of 3 reflections}
First, recall that a \emph{glide reflection} is an isometry formed by a reflection in a euclidean line (the \emph{axis} of the glide reflection) and a translation parallel to this line (the order of execution doesn't matter, since the two operations commute).
We begin by showing that the sandwich operator generated by the sum of a 1-vector and a 3-vector (line and pseudoscalar) corresponds to a glide reflection along the line.  Let $\vec{r} = \langle \vec{r} \rangle_{1} + \langle \vec{r} \rangle_{3} =\vec{m}  + \lambda \eye$ where $\vec{m}$ is normalized. Then for an arbitrary line $\vec{x}$:
\begin{align*}
\overline{\vec{r}}(\vec{x}) &= \vec{r} \vec{x} \widetilde{\vec{r}} \\
&= (\vec{m} + \lambda \eye) \vec{x} (\vec{m} - \lambda \eye) \\
&= \vec{m} \vec{x} \vec{m} + \vec{m}\vec{x}\lambda \eye - \lambda \eye \vec{x} \vec{m} - \lambda^{2}\eye^{2}\\
&= \vec{m} \vec{x} \vec{m} + \lambda \vec{m} \vec{x}^{\perp} - \lambda \vec{x}^{\perp} \vec{m} \\
&= \overline{ \vec{m}}(\vec{x}) + \lambda(\vec{m}\vec{x}^{\perp} - \vec{x}^{\perp}\vec{m})\\
&= \overline{ \vec{m}}(\vec{x}) + 2 \lambda(\vec{m} \cdot \vec{x}^{\perp}) \\
&=  \overline{ \vec{m}}(\vec{x})  + 2 \lambda (\cos{\alpha})\vec{\omega} 
\end{align*}
\begin{figure}
  \centering
    \setlength\fboxsep{0pt}\fbox{\includegraphics[width=0.66\textwidth]{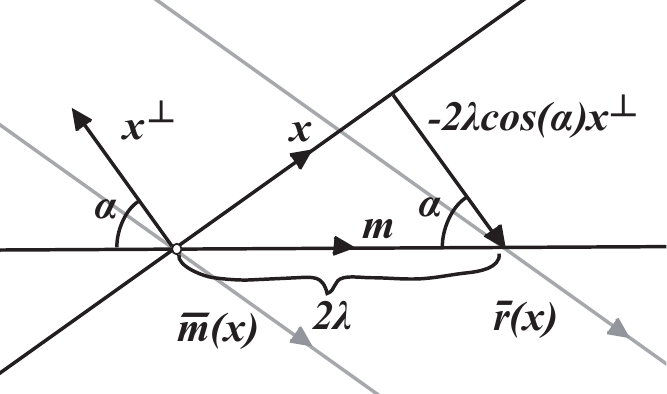}}
  \caption{Glide reflection generated by $\vec{r} = \vec{m} + \lambda \eye$ applied to line $\vec{x}$.}
\label{fig:glideref}
\end{figure}
The steps in the calculation follow from the discussion of the 2-way products above.
The result consists of two terms. The first term is the reflection of $\vec{x}$ in the line $\vec{m}$; by \Sec{sec:sumdiff} above, the second term represents the translation of the reflected line perpendicular to its own direction by the distance $2\lambda \cos(\alpha)$.   
The translation component reveals itself more clearly by considering $\overline{\vec{r}}(\vec{X})$ for an arbitrary \emph{point} $\vec{X}$.  A calculation similar to the above yields:
\begin{align*}
\overline{\vec{r}}(\vec{X}) &= ... = \overline{ \vec{m}}(\vec{X}) + 2 \lambda(\vec{m} \wedge \vec{X}^{\perp}) \\
&=  \overline{ \vec{m}}(\vec{X})  + 2 \lambda (\vec{m} \wedge \vec{\omega}) \\
&=  \overline{ \vec{m}}(\vec{X})  + 2 \lambda (\vec{m}_{\infty}) 
\end{align*}
In this form it is clear that the translation component is $2\lambda \vec{m}_{\infty}$: a translation in  the direction  of the line $\vec{m}$ through a distance $2\lambda$.  Consult \Fig{fig:glideref}.

Applying this to the situation of 3 reflections:  By \Sec{sec:threelns} above, the product of three lines has the form $\vec{r} = \vec{a}\vec{b}\vec{c} = \overline{\vec{b}} + \sin{(\alpha)}d_{\vec{a}\vec{A}}\eye$, hence the above results can be applied. Recall that $\overline{\vec{b}}$ is the joining line of $\overline{\vec{A}}$ and $\overline{\vec{C}}$, the feet of the altitudes from $\vec{A}$ and $\vec{C}$, resp. Refer to \Fig{fig:trianglelns}.  
%

\subsection{Exponential form for direct isometries}

It's not necessary to write a rotator as the product of two lines.  If one knows the desired angle of rotation, one can generate the rotor directly from the fixed point $\vec{P}$ of the rotation.   We know that it is normalized so that $\vec{P}^{2}=-1$.  Then, using a well-known technique of geometric algebra, one looks at the exponential power series $e^{t\vec{P}}$ and shows, in analogy to the case of complex number $i^{2}=-1$, that $e^{t\vec{P}} = \cos{t} + (\sin{t}) \vec{P}$. The right-hand side we already met above as the product of two euclidean lines meeting in the point $\vec{P}$ at the angle $t$.  Setting $t=\alpha$ one obtains the rotor $\vec{r}$ from the previous paragraphs.  What's more, letting $t$ take values from $0$ to $\alpha$ one obtains a smooth interpolation between the identity map and the desired rotation. Note that this sandwich operator rotates through the angle $2 \alpha$; to obtain a rotation of $\alpha$ around $\vec{P}$, set $\vec{r} = e^{\frac{\alpha \vec{P}}{2}}$. 

\myexercise  Carry out the same analysis for an ideal point $\vec{V}$ to obtain an exponential form for a translator that moves a distance $d$ in the direction perpendicular (CCW) to $\vec{V}$. [Answer: $e^{\frac{d \vec{V}}{2}} = 1+\frac{d}{2}\vec{V}$.] 
}
{
Now we have introduced the geometric product in detail, we recall the main facts about implementing euclidean isometries  in $\pdclal{2}{0}{1}$.  
\begin{compactenum}
\item \textbf{Reflections.}  All isometries are generated by the reflections in euclidean lines.  The reflection in the line $\vec{a}$ is given by the sandwich $\overline{\vec{a}}(\vec{X}) : =  \vec{a}\vec{X} \vec{a}$, where $\vec{X}$  is a $k$-vector of any grade.
\item \textbf{Rotors.} The composition of an even number of reflections is a direct isometry, either a rotation or a translation.  The product of an even number of $1$-vectors  is called a \emph{rotor} and gives rise to a direct isometry  by the sandwich operator $\overline{\vec{g}}(\vec{X}) : =  \vec{g}\vec{X} \widetilde{\vec{g}}$. 
\item \textbf{Rotations.}  When the two lines meet in a euclidean point $\vec{P}$ and make an angle $\alpha$, then $\overline{\vec{g}}$ is a euclidean rotation around $\vec{P}$ with angle $2 \alpha$.   
\item \textbf{Translations.} When the two lines meet in an ideal point $\vec{V}$ and are a distance $d$ apart, then $\overline{\vec{g}}$ implements a euclidean translation perpendicular to $\vec{V}$ through a distance $2d$.  
\item \textbf{Exponential form.} For a rotational rotor,  $\vec{g} = e^{{\alpha \vec{P}}} = \cos{(\alpha)} + \sin{(\alpha)} \vec{P}$; for a translational one, $\vec{g} = e^{{d \vec{U}}} = 1 + d \vec{U}$.
\end{compactenum}
\myexercise  Every rotor can be written as the product of two lines. 

}

\section{Orthogonal projections and rejections}
\label{sec:orthpro}
When one has two geometric entities it is often useful to be able to express one in terms of the other. Orthogonal projection is one method to obtain such a decomposition. 
For example, in the familiar euclidean VGA $\clal{3}{0}{0}$, any vector $\vec{b}$ can be decomposed with respect to a second vector $\vec{a}$ as $\vec{b} = \alpha \vec{a} + \beta \vec{a}^\perp$ where $\alpha, \beta \in \mathbb{R}$ and $\vec{a}^\perp \cdot \vec{a} = 0$. These two terms are sometimes called the \emph{projection}, resp., \emph{rejection} of $\vec{b}$ with respect to $\vec{a}$.
The algebra $\pdclal{2}{0}{1}$ offers a variety of such decompositions which we now discuss, both for their utility as well as to gain practice in using the geometric product introduced above.  We can project a line onto a line or a point; and a point onto a line or a point.  As before all points and lines are assumed to be normalized.    Consult \Fig{fig:orthProj}. 

\fvonly{Each projection follows the same pattern: take a product of the form $\vec{X}\vec{Y}\vec{Y}$ and apply associativity to obtain $\vec{X}(\vec{Y}\vec{Y}) = (\vec{X}\vec{Y})\vec{Y}$.  Assuming normalized arguments, $\vec{Y}\vec{Y} = \pm1$, yielding $\vec{X} = \pm(\vec{X}\vec{Y})\vec{Y}$.  The right-hand side typically consists of two terms representing an orthogonal decomposition of the left-hand side.    Note that,  like the reflection in a line (in which the first and last factors are identical),  such projections can be considered as a special form of a 3-way product,  in which either the first two or the last two factors are identical.}  

\subsection{Orthogonal projection of a line onto a line}

Assume both lines are euclidean.
Multiply the equation $\vec{m}\vec{n} = \vec{m}\cdot \vec{n} + \vec{m} \wedge \vec{n}$ with $\vec{n}$ on the right and use $\vec{n}^{2}=1$ to obtain 
\begin{align*}
\vec{m} &= (\vec{m} \cdot \vec{n})\vec{n} + (\vec{m} \wedge \vec{n}) \vec{n} \\
&= (\cos{\alpha}) \vec{n} + (\sin{\alpha} )\vec{P}\vec{n} \\
&= (\cos{\alpha}) \vec{n} - (\sin{\alpha}) \vec{n}^{\perp}_{\vec{P}}
\end{align*}
Note that $\vec{P}\vec{n} =-\vec{n}^{\perp}_{\vec{P}}$ since $\vec{P}\wedge\vec{n} = 0$. Thus one obtains a decomposition of $\vec{m}$ as the linear combination of $\vec{n}$ and the perpendicular line $\vec{n}^{\perp}_{\vec{P}}$ through $\vec{P}$.  See \Fig{fig:orthProj}, left.

\myexercise  If the lines are parallel one obtains $\vec{m} = \vec{n} + d_{\vec{m}\vec{n}}\omega$.  

 \begin{figure}
  \centering
{\setlength\fboxsep{0pt}\fbox{\includegraphics[width=.32\columnwidth]{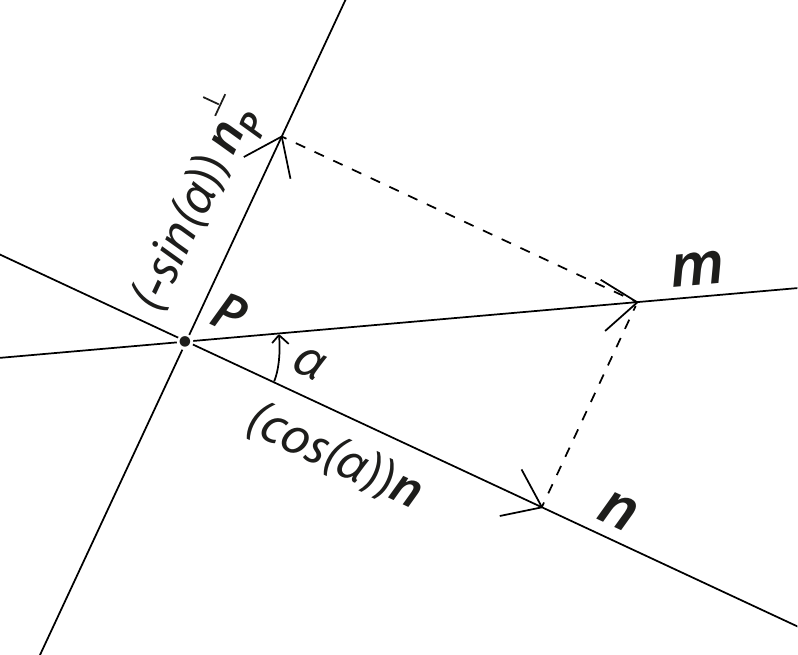}}}\hspace{.02in}
{\setlength\fboxsep{0pt}\fbox{\includegraphics[width=.32\columnwidth]{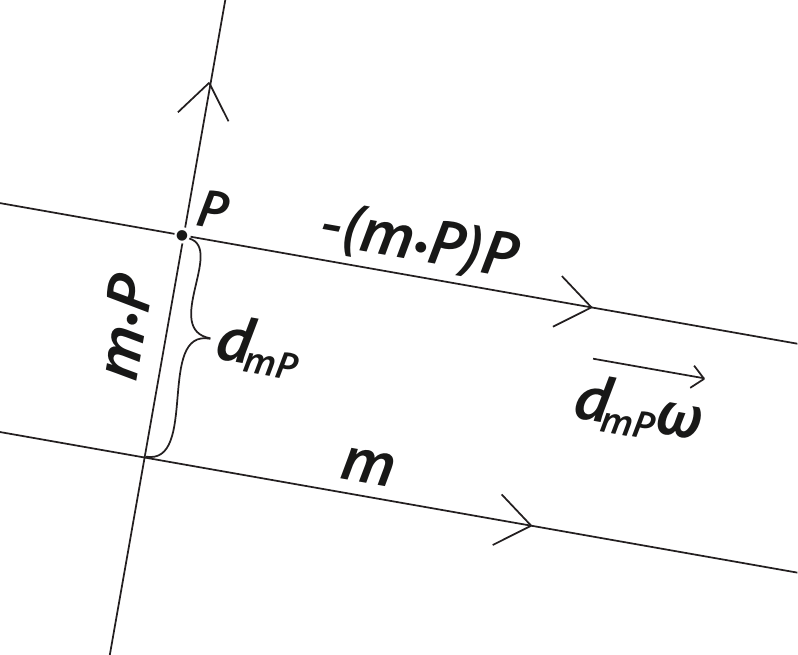}}}\hspace{.02in}
{\setlength\fboxsep{0pt}\fbox{\includegraphics[width=.32\columnwidth]{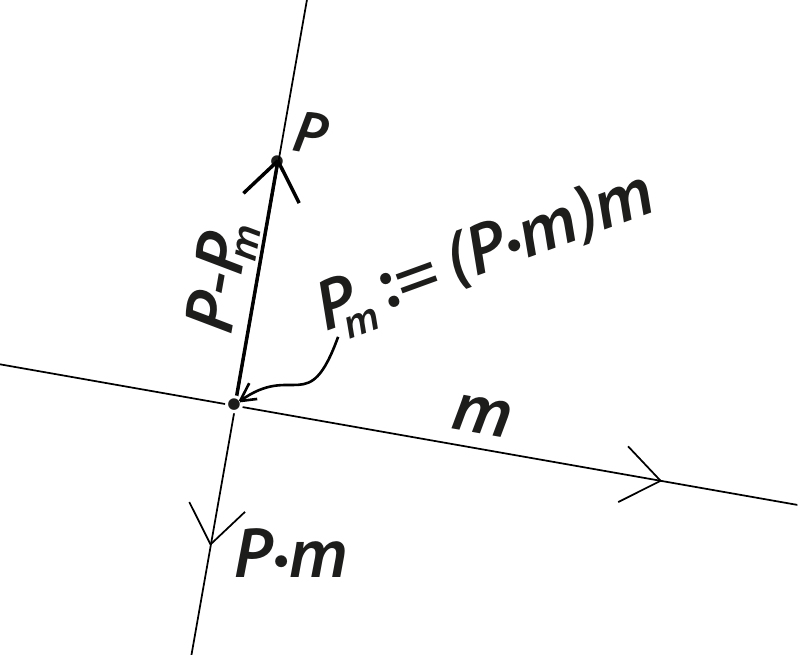}}}
\caption{Orthogonal projections (l. to r.): line $\vec{m}$ onto line $\vec{n}$,  line $\vec{m}$ onto point $\vec{P}$, and point  $\vec{P}$  onto line $\vec{m}$.}
\label{fig:orthProj}
\end{figure}


\subsection{Orthogonal projection of a line onto a point}

Assume both point and line are euclidean.
Multiply the equation $\vec{m}\vec{P} = \vec{m}\cdot \vec{P} + \vec{m} \wedge \vec{P}$ with $\vec{P}$ on the right and use $\vec{P}^{2}=-1$ to obtain 
\begin{align*}
\vec{m} &= -(\vec{m} \cdot \vec{P})\vec{P} - (\vec{m} \wedge \vec{P}) \vec{P} \\
&= -\vec{m}^{\perp}_{\vec{P}}\vec{P} - (d_{\vec{m}\vec{P}}\eye) \vec{P} \\
&= \vec{m}^{||}_{\vec{P}} - d_{\vec{m}\vec{P}} \vec{\omega} \\
\end{align*}
In the third equation, $\vec{m}^{||}_{\vec{P}}$ is the line through $\vec{P}$ parallel to $\vec{m}$, with the same orientation.  Thus one obtains a decomposition of $\vec{m}$ as the sum of a line through $\vec{P}$ parallel to $\vec{m}$ and a multiple of the ideal line (adding which, as noted above in \Sec{sec:sumdiff}, translates euclidean lines parallel to themselves). See \Fig{fig:orthProj}, middle.

\vspace{-.05in}

\subsection{Orthogonal projection of a point onto a line}

Assume both point and line are euclidean.
Multiply the equation \[\vec{m}\vec{P} = \vec{m}\cdot \vec{P} + \vec{m} \wedge \vec{P}\] on the left with $\vec{m}$ on the right and use $\vec{m}^{2}=1$ to obtain 
\begin{align*}
\vec{P} &=  \vec{m}(\vec{m} \cdot \vec{P})+ (\vec{m} \wedge \vec{P}) \\
&= \vec{m} (\vec{m}^{\perp}_{\vec{P}})+  \vec{m}(d_{\vec{m}\vec{P}}\eye) \\
&= \vec{P}_{m} + d_{\vec{m}\vec{P}}\vec{m}^{\perp} \\
&= \vec{P}_{m} + (\vec{P} - \vec{P}_{m})
\end{align*}
In the third equation, $\vec{P}_{m}$ is the point of $\vec{m}$ closest to $\vec{m}$.  The second term of the third equation is a vector perpendicular to $\vec{m}$ whose length is $d_{\vec{P}\vec{m}}$:  exactly the vector $\vec{P} - \vec{P}_{m}$.
 Thus one obtains a decomposition of $\vec{P}$ as the point on $\vec{m}$ closest to $\vec{P}$   plus a vector perpendicular to $\vec{m}$. See \Fig{fig:orthProj}, right.

\myexercise  Show that the orthogonal projection of a euclidean point $\vec{P}$ onto another euclidean point $\vec{Q}$ yields $\vec{P} = \vec{Q} + (\vec{P} - \vec{Q})$. 

\section{Worked-out example of euclidean plane geometry}
\label{sec:example}
We pose a problem in euclidean plane geometry on which to practice  the theory developed up to now:
\begin{quote}
Given a point $\vec{A}$ lying on an oriented line $\vec{m}$, and a second point $\vec{A}'$ lying on a second oriented line $\vec{m}'$, construct the unique direct isometry mapping $\vec{A}$ to $\vec{A}'$ and $\vec{m}$ to $\vec{m}'$.
\end{quote}
 \begin{figure}[b]
 \centering
{\includegraphics[width=.3\columnwidth]{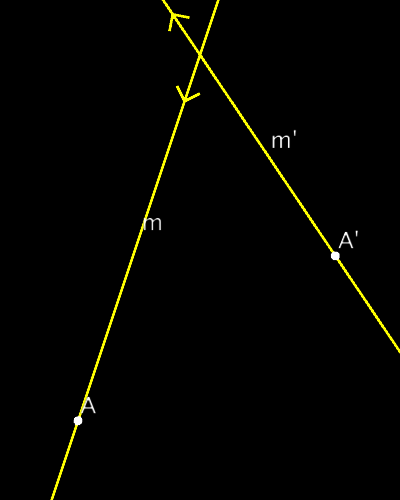}}\hspace{.1in}
{\includegraphics[width=.3\columnwidth]{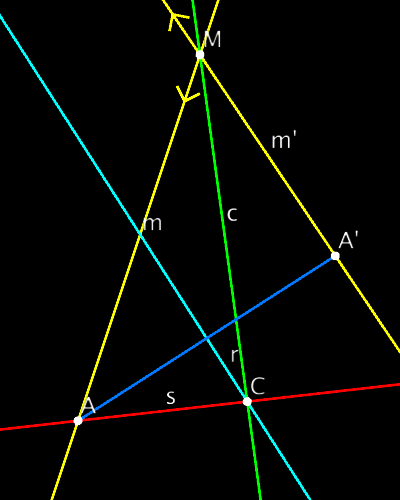}}\hspace{.1in}
{\includegraphics[width=.3\columnwidth]{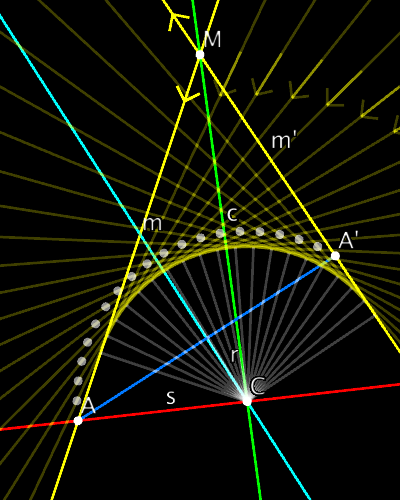}}
\caption{Left to right: the problem setting, the solution, interpolating the solution.}
\label{fig:ptlnpair}
\end{figure}
The problem is illustrated in \Fig{fig:ptlnpair} (left), including orientation on the two lines.  We assume the points and lines are normalized, and define to begin with the intersection point of the lines and the joining line of the points:
\begin{align*} 
\mathbf{M} &:= \mathbf{m} \wedge \mathbf{m'},\qquad
\mathbf{a} := \mathbf{A} \vee \mathbf{A'}   
\end{align*}
The direct isometry we are seeking is either a rotation or a translation.  In the former case, the center of rotation has to be equidistant from $\vec{A}$ and $\vec{A}'$, that is, it lies on the perpendicular bisector of the segment $\overline{\vec{A}\vec{A}'}$. To construct this we first obtain the midpoint, and then, applying \Sec{sec:prlnpt}, construct the perpendicular line through the midpoint:
\begin{align*}
\mathbf{A_m} &:= \mathbf{A} + \mathbf{A'}, \qquad
\mathbf{r} := \mathbf{A_m} \cdot \mathbf{a}~~ (=\mathbf{A_m} \mathbf{a})   
\end{align*}
The condition that $\vec{m}$ maps to $\vec{m}'$ implies that the center of rotation is the same distance from $\vec{m}$ as from $\vec{m}'$, that is, lies on the angle bisector of the two lines.  We choose the difference in order to respect the orientations of the lines, as the reader can readily confirm. The desired center is then the intersection $\vec{C}$ of $\vec{r}$ and $\vec{c}$.
\begin{align*}
\mathbf{c} &:= \mathbf{m} - \mathbf{m'}, \qquad  
\mathbf{C} := \mathbf{r} \wedge \mathbf{c}   
\end{align*}
The final step is to construct the desired isometry.  We can (for a rotation) find two lines through $\vec{C}$ that meet at half the desired angle of rotation: the line $\vec{A} \vee \vec{C}$ and the perpendicular bisector $\vec{r}$ satisfy this condition.  Then form the rotor of their product; the rotation is then the sandwich operator defined by this rotor.  
\begin{align*}
\mathbf{s} &:= \mathbf{A} \vee \mathbf{C}, \qquad 
\mathbf{g} := \mathbf{r}  \mathbf{s}, \qquad
\overline{\mathbf{g}}(\mathbf{X})  := \vec{g} \mathbf{X} \widetilde{\vec{g}}  
 \end{align*}
 One can also calculate the angle $\alpha$ between the two mirror lines from the equation $\cos{{\alpha}} = \vec{r} \cdot \vec{s}$, and use this to calculate $\vec{g}$ as an exponential: $\vec{g} = e^{\alpha \vec{C}}$. 

\myexercise  Show that the above construction also yields valid results when $\vec{C}$ is ideal, and that the resulting isometry is a translation. 

\section{Directions for further study}
\label{sec:neg}

\def\MN{metric-neutral\xspace}

For readers who are intrigued by the approach presented here, there are several natural directions for further study.  If one wants to stay within plane geometry, there are many themes that could be cast into the PGA format. For example,  one could explore calculus and differentiation in the plane, including point-wise and line-wise curves, point- and line-valued functions, etc.  For a general introduction to differentiation in geometric algebra see \cite{dfm07}, Ch. 8.  This could lead to a treatment of 2D kinematics and rigid body dynamics.  Or, one could use the discussion of three-way products in \Sec{sec:threes} as a starting point for formulating the theory of triangles and triangle centers in this language.  One practical direction would be to apply the theory sketched here as a framework for 2D graphics programming.

Another natural direction  is to move from 2 to 3 dimensions and explore the euclidean PGA $\pdclal{3}{0}{1}$ for euclidean 3-space $\Euc{3}$. Available resources include \cite{gunnThesis} (Ch. 7), \cite{gunn2011}, and \cite{gunnFull2010}.  While many results presented here generalize without surprises to 3 dimensions, one conceptual challenge presented in moving to 3 dimensions is that the space of bivectors, crucial to kinematics and dynamics, is no longer exhausted by the simple bivectors (which in this case represent lines in 3-space); the non-simple bivectors, known classically as \emph{linear line complexes}, exhibit much more complex -- and interesting -- behavior. An exhaustive treatment of the geometric product modeled on the one presented in the first half of this article would accordingly yield a richer, more complicated picture.

  \begin{figure}[t]
   \centering
{\setlength\fboxsep{0pt}\fbox{\includegraphics[width=.6\columnwidth]{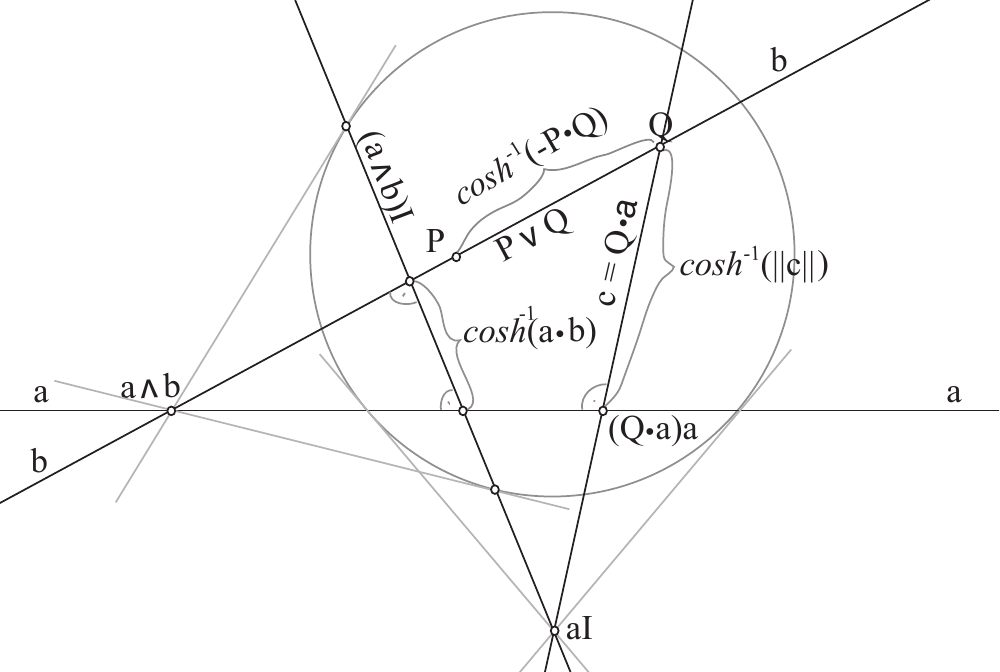}}}
\caption{Using $\pdclal{2}{1}{0}$ to do hyperbolic plane geometry.}
\label{fig:hypplane}
\end{figure}

Practitioners of non-euclidean geometry may be interested to know that the approach outlined here for the euclidean plane  can be carried out analogously for the hyperbolic and elliptic planes using the algebras $\pdclal{3}{0}{0}$, \textit{resp}., $\pdclal{2}{1}{0}$.\footnote{We favor using the dual construction here also (even though it is not strictly required) since then reflections in lines are represented by sandwiches with 1-vectors.  In the standard approach, where 1-vectors are points, such sandwiches represent the less familiar, less practical \quot{point reflections}.} 
Most of the features discussed above for the euclidean plane have non-euclidean analogies which possess a similar elegance and succinctness. An introduction to these metric planes is given in Ch. 6 of \cite{gunnThesis}, from which \Fig{fig:hypplane} is taken. This presents a \emph{\MN} approach, that is, results are stated whenever possible without specifying the metric. 

\vspace{-.1in}

\section{Evaluation and conclusion}
\label{sec:eac}
We  have shown that traditional euclidean plane geometry can be formulated in a compact and elegant form using $\pdclal{2}{0}{1}$.  We have successfully applied the algebra  to a variety of practical problems of plane geometry and have encountered no obstacles to the program of extending it to all aspects of euclidean plane geometry.  

How do these results compare to existing approaches?  Plane geometry is usually handled with a mixture of analytic geometry, linear algebra, and vector algebra.  The foregoing has established that  $\pdclal{2}{0}{1}$ offers a variety of desirable \quot{infrastructure} features which this mixed approach does not offer:
\begin{compactenum}
\item It is coordinate-free (for details see Appendix \ref{sec:cfd}).
\item Points and lines are equal citizens, rather than  lines being defined in terms of points.
\item \emph{Ideal} elements are integrated organically, both in incidence (intersection of parallels) and metric relations.
\item Join and meet operators are obtained from the Grassmann algebra.
\item Isometries are represented by versor sandwich operators that act uniformly on primitives of all grades.  The rotors have an exponential representation.
\item The geometric product provides a rich, interrelated family of formulas for distance and angle integrating seamlessly both euclidean and ideal elements.
\end{compactenum}

The last point above reflects a novel feature of $\pdclal{2}{0}{1}$ of special note: euclidean and ideal elements are tightly interwoven in an organic whole.  See for example the discussion of the 2-way products in \Sec{sec:gpdetail} and the collection of formulas in \Sec{sec:daf}. This tight integration is, to the best of our knowledge, available nowhere else.  We think it deserves to be better known and understood.  The discussion in \Sec{sec:eipfromiip}  below makes a modest start towards a deeper  understanding.

Implementing this algebra within modern programming languages presents no significant challenges.  The author has implemented it in Java, JavaScript, and Mathematica (at different times, for different purposes) and  successfully applied the resulting toolkit to a variety of practical geometric and graphical problems. The resulting infrastructure gains, in comparison to traditional approaches, have been gratifying.

How does  $\pdclal{2}{0}{1}$ compare to the other two geometric algebras mentioned at the beginning of the article?  \cite{calvet2007} is a treatment of plane geometry based on $\clal{2}{0}{0}$.  While entirely appropriate as an introduction to GA at the high school level, it  makes extensive use of non-GA techniques to overcome the limitations of $\clal{2}{0}{0}$, which unlike the euclidean plane contains a distinguished point (the origin), and can by itself model neither parallelism nor translations.  One of the leitmotifs of this article has been to show how $\clal{2}{0}{0}$ is embedded organically within  $\pdclal{2}{0}{1}$ as the ideal line $ \vec{\omega}$, so all the features of $\clal{2}{0}{0}$ can be accessed easily in the model presented here.
We are not aware of an analogous treatment of plane geometry in CGA to the one presented here.  \cite{gunn2016} provides a general comparison of CGA and PGA for euclidean geometry and establishes that for flat geometric primitives, such as the domain of classical plane geometry treated in this article, PGA displays advantages over CGA with regard to robustness, simplicity of representation, and ease of learning.

To sum up:  we have demonstrated that the model of plane euclidean geometry provided by PGA is complete, compact, computable, and elegant.  Whether considered pedagogically, practically, or scientifically, we believe PGA provides a viable alternative to traditional approaches to euclidean plane geometry.   By helping to modernize the teaching of euclidean geometry, it could make an important contribution to the task mentioned at the beginning of the article, of bringing the dramatic advances in 19th century mathematics in geometry and algebra to a wider audience.  


\appendix
\renewcommand{\appendixtocname}{Appendix}
\appendixpage
\section{Coordinate-free description}
\label{sec:cfd}
We provide here a modern, coordinate-free description of the algebra instead of the more traditional coordinate-based approach used above in \Sec{sec:bva} and \Sec{sec:npl}.  

\subsection{Foundations}
Let ${V}$ be a real, 3-dimensional vector space with dual space ${V}^{*}$.  We construct a geometric algebra $\textsf{A}$ based on ${V}$ using the signature $(2,0,1)$. We describe it here algebraically, and postpone until later the geometric interpretation. We begin by recalling some basic facts and definitions regarding the underlying Grassmann algebra \textsf{G} based on $V$:
\begin{compactenum}[$\bullet$]
\item \textsf{G} is a graded algebra consisting of 4 grades:
\begin{compactenum}[$\cdot$]
\item  $\bigwedge^{0}({V})$ is the 1-dimensional subspace of scalars $\mathbb{R}1$.
\item  $\bigwedge^{1}({V})$ can be identified with ${V}$.
\item $\bigwedge^{2}({V})$ can be identified with ${V}^{*}$. 
\item $\bigwedge^{3}({V})$ is a 1-dimensional vector space of pseudoscalars $\mathbb{R}\eye$. $\eye$ is defined more precisely below in \Sec{sec:neb}. 
\end{compactenum}
\item An element of  $\bigwedge^{k}({V})$ is called a $k$-vector.
\item There is an anti-symmetric bilinear product $\wedge$ (called the \emph{wedge} or \emph{Grassmann} product) defined on \textsf{G} that mirrors the subspace structure of weighted subspaces of $V$. For $\vec{a} \in   \bigwedge^{k}({V})$ and $\vec{b} \in  \bigwedge^{m}({V})$, 
\begin{compactenum}[$\cdot$]
\item  $\vec{a}\wedge\vec{b}$ = 0 $\iff$ $\vec{a}$ and $\vec{b}$ are linearly dependent.
\item  Otherwise, $\vec{a}\wedge\vec{b} \in \bigwedge^{k+m}({V})$ represents the weighted subspace spanned by $\vec{a}$ and $\vec{b}$.
\end{compactenum}
\item Let $\vec{a}\in \bigwedge^{1}({V})$ and $\vec{A} \in \bigwedge^{2}({V})$.   We say $\vec{a}$ and $\vec{A}$ are \emph{incident} $\iff \vec{a} \wedge \vec{A}= 0$.
\item For a vector subspace ${T} \subset \bigwedge^{k}({V})$ define the \emph{outer product null space} ${T}_{\wedge}^{\perp} := \{\vec{x} \in \bigwedge^{3-k}({V})\mid \vec{t} \wedge \vec{x} = 0~ \forall~\vec{t} \in {T}\}$.
\item Notation: For a multi-vector $\vec{M} \in \textsf{G}$, $\vec{M} = \sum_{k}\grade{\vec{M}}{k}$ where $\grade{\vec{M}}{k}$ is the grade-$k$ part of $\vec{M}$.
\end{compactenum}

\subsection{Euclidean and ideal elements}  
The inner product of the geometric algebra can be represented by a symmetric bilinear form $B : \vsp \otimes \vsp\rightarrow \mathbb{R}$. 
The \emph{kernel} of $B$ is defined as:  \[N := \{\vec{n} \in \vsp \mid B(\vec{n},\vec{x}) = 0~~\forall ~\vec{x}\}\]   
The signature of the inner product is $(2,0,1)$.  The 1 in the third position gives the dimension of $N$.  So, $N$ is a 1-dimensional vector sub-space of $\vsp$.  As such, it is generated by an element $\vec{\omega}$, which we will specify more precisely below in \Sec{sec:neb}.  Elements of $N$ are called \emph{ideal} vectors. 
Vectors not in $N$ are called \emph{euclidean} (or \emph{proper}).  $N^{\perp}_{\wedge}$ consists of bivectors incident with $\vec{\omega}$, and is a 2-dimensional subspace of $\bigwedge^{2}(\vsp)$.  
An element of $N^{\perp}_{\wedge}$ is said to be an \emph{ideal} bivector; all other bivectors are euclidean (or proper).

\subsubsection{The square of a 1-vector; normalized euclidean vectors}

In a geometric algebra, the geometric product is defined on 1-vectors by \[\vec{a}\vec{b} = \vec{a}\cdot \vec{b} + \vec{a}\wedge \vec{b}\] where $ \vec{a}\cdot \vec{b} = B(\vec{a},\vec{b})$ and $\vec{a}\wedge \vec{b}$ is the the exterior product of the underlying Grassmann algebra. 
The geometric product $\vec{m}^{2}$ for a 1-vector $\vec{m}$ reduces to $\vec{m} \cdot \vec{m}$ since the wedge product is antisymmetric.  For $\vec{m} = \vec{\omega}$, $\vec{\omega}^{2} = \vec{\omega}\cdot \vec{\omega} = 0$ since $\vec{\omega} \in N$.  For any euclidean vector $\vec{m}$, \[ \vec{m}^{2} = \vec{m} \cdot \vec{m} = k \in \mathbb{R}^{+}\]  We define the norm $\| \vec{m} \| := \sqrt{\vec{m}^{2}}$.  Then $\vec{m}_{n} := \sqrt{k}^{-1}\vec{m}$ satisfies $\| \vec{m}_{n} \| = 1$; such a vector is said to be \emph{normalized}. 

\subsubsection{The square of a 2-vector}
\label{sec:neb}
From the above, there are two sorts of  bivectors, ideal and euclidean.  
 For  ideal $\vec{U}$,  $\vec{U} = \vec{\omega}~ \wedge ~\vec{m} $ for some euclidean vector $\vec{m}$. And, since $\vec{\omega} \in N$, $\vec{\omega}~ \wedge ~\vec{m}= \vec{\omega} \vec{m}$.  Then  $\vec{U}^{2} = -\vec{\omega}^{2} \vec{m}^{2} = 0$.   Using the following exercise, it is easy to calculate that $\vec{P}^{2} = -1$. 
 Hence a bivector is ideal $\iff$ its square is zero.  
 
 \myexercise  For normalized euclidean $\vec{P}$, one can find two orthonormal euclidean 1-vectors $\vec{m}$ and $\vec{n}$ such that $\vec{P} = \vec{m}\vec{n}$. 

\subsubsection{Normalized euclidean 2-vectors}
We could define a normalized euclidean bivector to be a bivector satisfying  $\vec{P}^{2}=-1$.  But we can do better, as the following discussion shows.  
Let $\vec{P}$ be any euclidean 2-vector satisfying $\vec{P}^{2}=-1$. Recall the definition of the scaled magnitude function S in \Sec{sec:bva}. We fix $\vec{\omega}$ to be the unique element of $N$ satisfying $S(\vec{\omega} \wedge \vec{P}) = 1$, and define $\eye := \vec{\omega} \wedge \vec{P}$. We show that these definitions don't depend on $\vec{P}$, and that the value of $S(\vec{\omega} \wedge \vec{P})$ can serve as a norm for bivectors.

\begin{mylemma}
For euclidean bivector $\vec{P}$ and ideal bivector $\vec{U}$, $ \grade{\vec{P}\vec{U}}{0} = 0$.
\end{mylemma}
\begin{proof}
Choose $\vec{m} \in \vec{U}^{\perp}_{\wedge} \cap \vec{P}^{\perp}_{\wedge}$ with $\| \vec{m} \| = 1$.\footnote{Or define $\vec{m} = \vec{U} \vee \vec{P}$ and normalize $\vec{m}$.}  Then $\vec{U} =  \lambda \vec{m} \vec{\omega} $ for $\lambda \in \mathbb{R}^{*}$. Write $\vec{P} =  \vec{n}\vec{m}$ where $\vec{n}$ is normalized and orthogonal to $\vec{m}$.  Then
\begin{align*}
{\vec{P}\vec{U}}&={(\vec{n}\vec{m})(\lambda\vec{m}\vec{\omega})} \\
&=  \lambda{\vec{n}(\vec{m}^{2})\vec{\omega}} \\
&=  \lambda{\vec{n}\vec{\omega}} 
\end{align*}
Here we have used associativity of the geometric product, and the fact that $\vec{m}$ is normalized.  Finally, since $\vec{\omega} \in N$,   $\grade{\vec{n}\vec{\omega}}{0}  = \vec{n} \cdot \vec{\omega} = 0$.
\end{proof}
\begin{mylemma}
Given  euclidean bivectors $\vec{P}$ and $\vec{Q}$, $\vec{Q}  = \lambda \vec{P} + \vec{U}$ for some $ \lambda \in \mathbb{R}, \lambda \neq 0$ and $\vec{U} \in N_{\wedge}^{\perp}$. Furthermore, $\vec{Q}^{2} = \lambda^{2}\vec{P}^{2}$.
\end{mylemma}
\begin{proof}
The first part follows by observing that $N^{\perp}_{\wedge} \subset \bigwedge^{2}(V)$ is a subspace of co-dimension 1 in $\bigwedge^{2}(V)$, and $\vec{Q}, \vec{P} \notin N^{\perp}_{\wedge} $. The second assertion follows by observing:
\begin{align*}
\vec{Q}^{2} &= (\lambda \vec{P} + \vec{U})^{2}\\
&=\lambda^{2} \vec{P}^{2} + \lambda(\vec{P}\vec{U} + \vec{U}\vec{P}) + \vec{U}^{2}\\
&=\lambda^{2}  \vec{P}^{2} + 2 \lambda \grade{\vec{P}\vec{U}}{0} \\
&= \lambda^{2} \vec{P}^{2}
\end{align*} 
Here we have used the fact that the grade-0 part of the geometric product $ \vec{P}\vec{U}$ is the symmetric part of the product, that $\vec{U}^{2}=0$ for ideal $\vec{U}$, and the previous lemma.  
\end{proof}
\begin{mythm}
Given euclidean bivectors $\vec{P}$ and $\vec{Q}$ such that $\vec{P}^{2} = \vec{Q}^{2}$, and $\vec{\omega} \in N$. Then $\vec{\omega} \wedge \vec{P}=\pm \vec{\omega} \wedge \vec{Q}$.  
\end{mythm}
\begin{proof}
By Lemma 2, $\vec{Q} = \lambda \vec{P} + \vec{U}$ for $\vec{U} \in N_{\wedge}^{\perp}$.  Since $\vec{Q}^{2} = \vec{P}^{2}$, $\lambda = \pm 1$.  Wedging with $\vec{\omega}$ yields $\vec{\omega}\wedge \vec{Q} = \pm \vec{\omega} \wedge \vec{P} + \vec{\omega} \wedge \vec{U} = \pm \vec{\omega} \wedge \vec{P}$.
\end{proof}

The preceding theorem allows us to obtain a stronger normalization than the condition $\vec{Q}^{2}=-1$.  Define the norm of a bivector to be $\| \vec{Q} \| := S(\vec{\omega} \wedge \vec{Q})$.  We say that a euclidean bivector $\vec{Q}$ is \emph{normalized} when $\| \vec{Q} \| = 1$.    In every one-dimensional vector subspace of $\bigwedge^{2}(V)$, there are two solutions $\{ \vec{Q}, -\vec{Q}\}$ to $\vec{Q}^{2}=-1$. $\| \vec{Q} \| = 1$ picks out exactly one of these solutions. The uniqueness of this result simplifies many calculations.

\subsubsection{Multiplication by the pseudoscalar}
 Multiplication by the basis pseudoscalar $\eye$ is an important operation, sometimes called the \emph{polarity on the metric quadric}. It maps an element to its orthogonal complement with respect to the inner product.  This multiplication is important enough to merit its own notation \[\Pi(\vec{X}) := \vec{X}^{\perp} := \eye \vec{X}\] The result is  called the  \emph{polar} of $\vec{X}$.  By the previous section,  a euclidean bivector $\vec{Q}$ is normalized $\iff  \eye := \vec{\omega} \vec{Q}$.  Then the polar of a normalized euclidean point $\vec{P}$ is given by $ \vec{P}^{\perp} =  \eye \vec{P} =  -\vec{\omega}$ since $\vec{P}^{2} = -1$.  
 
 $\vec{X}^\perp$ is sometimes called the \emph{inner product null space} of $\vec{X}$.  Note in contrast that $\vec{X}^\perp_{\wedge}$ is the \emph{outer product null space}.
 
 The situation is a little more complicated for 1-vectors. Let $\vec{m}$ be a normalized euclidean 1-vector.  Let $\vec{n}$ be a 1-vector orthogonal to $\vec{m}$. Then the product $\vec{n}\vec{m}$ is a normalized euclidean 2-vector, hence $\eye = \vec{\omega}\vec{n}\vec{m}$ and $\vec{m}^{\perp} := \eye \vec{m} = \vec{\omega} \vec{n} = \vec{U}$, where $\vec{U}$ is an ideal bivector.   
 
 \myexercise  The kernel of $\Pi$, restricted to 1-vectors, is $N$, while the kernel of $\Pi$, restricted to 2-vectors, is $N_{\wedge}^{\perp}$. 

\subsection{Ideal inner product on ideal bivectors}

We saw above that euclidean bivectors can be normalized, but an ideal bivector $\vec{U}$ satisfies $\vec{U}^{2}=0$ hence cannot be normalized in the same way.  However, there is a way to define an alternative norm -- along with an associated inner product -- on the ideal bivectors.  We define this \emph{ideal} inner product and then  show how to derive the  complete inner product structure on the euclidean elements (of all grades) from this ideal inner product.  

\subsubsection{The quotient space $\vsp/\vec{\omega}$}
Define an equivalence relation on the set of  euclidean vectors: \[\vec{m} \equiv \vec{n} \iff~\exists~ c \in \mathbb{R}~\text{such that}~\vec{m} - \vec{n} = c\,\vec{\omega}.\]  Let the equivalence class of $\vec{m}$ be denoted by $[\vec{m}]$.   Define a symmetric bilinear form $\widetilde{B}$ on the resulting quotient space $\vsp/\vec{\omega}$ by $\widetilde{B}( [\vec{m}], [\vec{n}]) := B( \vec{m}, \vec{n})$.  This is well-defined.  For if $\widetilde{\vec{m}}$ and $\widetilde{\vec{n}}$ are two other representatives, then $\widetilde{\vec{m}} = \vec{m} + c\vec{\omega}$ and  $\widetilde{\vec{n}} = \vec{n} + d\vec{\omega}$.  $B(\widetilde{ \vec{m}}, \widetilde{\vec{n}}) = B(\vec{m} + c\vec{\omega}, \vec{n} + d\vec{\omega}) = B(\vec{m}, \vec{n})$ since $\vec{\omega} \in N$.  

\subsubsection{An \quot{ideal} inner product on $N^{\perp}_{\wedge}$}

Furthermore,  \[\vec{m} \equiv \vec{n} \iff \Pi(\vec{m}) = \Pi(\vec{n})\] $(\Rightarrow)$: If $\vec{m} \equiv \vec{n}$,  then $\vec{m} = \vec{n} + c\vec{\omega}$ and $\Pi(\vec{m}) = \Pi(\vec{n}) + c\Pi(\vec{\omega}) = \Pi(\vec{n})$ since $\vec{\omega} \in \ker(\Pi)$.  $(\Leftarrow)$: If $\Pi(\vec{m}) = \Pi(\vec{n})$, then by linearity $\Pi(\vec{m}) -\Pi(\vec{n}) = \eye(\vec{m} - \vec{n}) =  0$.  This means $\vec{m} - \vec{n} \in ker(\Pi)$.  Hence, by the  previous paragraph, $\vec{m} - \vec{n} = c\vec{\omega}$ for $c \in \mathbb{R}$. Thus $\widetilde{\Pi} : \vsp/\vec{\omega} \rightarrow N_{\wedge}^{\perp}$ defined by $\widetilde{\Pi}([\vec{m}]) := \Pi(\vec{m})$ is well-defined.  In fact we have shown that it is a bijection and hence has a well-defined inverse.

Use this inverse to transfer the inner product $\widetilde{B}( [\vec{m}], [\vec{n}])$ onto the ideal bivectors via \[ \langle \vec{U},\vec{W} \rangle_{\infty}  := \widetilde{B}(\widetilde{ \Pi}^{-1}(\vec{U}), \widetilde{\Pi}^{-1}(\vec{W}) )\]   It's not hard to show that $\langle , \rangle_{\infty}$ is the standard positive definite inner product  on $N^{\perp}_{\wedge}$ (since we began with the signature $(2,0,1)$) .  This induces a norm on ideal bivectors by $\| \vec{U} \|_{\infty} := \sqrt{\langle \vec{U}, \vec{U} \rangle_{\infty}}$.  It is always possible to choose a representative for $\vec{U}$ so that $\| \vec{U} \|_{\infty} = 1$. 

\subsection{Recreating the $(2,0,1)$ inner product from the ideal inner product}
\label{sec:eipfromiip}
 It's tempting to view the ideal norm $\langle, \rangle_{\infty}$ as something \emph{ad hoc} added on to the algebra $\pdclal{2}{0}{1}$. However,  the above discussion supports the contrary interpretation that the ideal inner product $\langle, \rangle_{\infty}$ on ideal bivectors is the \emph{primary} structure from which the inner product $(2,0,1)$ on vectors is derived, rather than vice-versa.  For, let $\vec{a}$ and $\vec{b}$ be two euclidean vectors, and $\vec{A} := \vec{a} \wedge \vec{\omega}$ and $\vec{B} := \vec{b} \wedge \vec{\omega}$ be their wedge product with the ideal vector.  Then define a symmetric bilinear form $\widehat{B} (\vec{a}, \vec{b}) := \langle \vec{A}, \vec{B} \rangle_{\infty}$.  From the above discussion it is clear that $\widehat{B}$ is  well-defined, and in fact, $\widehat{B} = B$.  So one can begin with the ideal bivector subspace $N^{\perp}_{\wedge}$ equipped with the signature $(2,0,0)$ and \quot{push} it in this straightforward way onto the euclidean 1-vectors to obtain the euclidean plane.  Similar constructions work for any dimension.

\subsection{Interpretation with respect to $\pdclal{2}{0}{1}$}

The above treatment has been carried out for an abstract real vector space $\vsp$ of dimension 3.  To arrive at the algebra $\pdclal{2}{0}{1}$ one must specify $\vsp$, as outlined in \Sec{sec:algintro} above, which leads to the choice  $\vsp := \RD{3}$, the dual space of $\R{3}$. In the resulting  vector space geometric algebra $\dclal{2}{0}{1}$, 1-vectors represents oriented planes through the origin and 2-vectors represent standard vectors.  In the second step, the algebra has to be projectivized to form $\pdclal{2}{0}{1}$.   Hence, 1-vectors transform to lines and 2-vectors become points.  In particular, $\vec{\omega}$ represents a plane in $\RD{3}$, and when projectivized represents a line, the \emph{ideal line} of the euclidean plane.   The ideal bivectors are ideal points, incident with $\vec{\omega}$.  Interpreting the contents of \Sec{sec:neb} in this light: the difference $\vec{P}-\vec{Q}$, for normalized euclidean points $\vec{P}$ and $\vec{Q}$, is an ideal point.  This is reminiscent of how free vectors are defined to be the difference of two euclidean points. In fact,  ideal points are equivalent to free vectors, an insight already made by Clifford in \cite{clifford73}, so that the vector algebra $\clal{2}{0}{0}$ is contained here as the ideal line with its ideal inner product.  

The equivalence classes of $\vsp/\vec{\omega}$, in the context of $\pdclal{2}{0}{1}$, are families of parallel lines, which share a common ideal point. Such a set of lines is known as a \emph{line pencil} in classical projective geometry; in this case the pencil is centered on (or \emph{carried by}) an ideal point.   To see this:  $\vec{m} \equiv \vec{n} \iff \vec{m} - \vec{n} = c\vec{\omega}$.  The  point $\vec{U} := \vec{m} \wedge \vec{n}$ satisfies $\vec{U} \wedge \vec{m} = \vec{U} \wedge  \vec{n} = 0$.  Hence $\vec{U} \wedge \vec{\omega} = 0$, which shows that $\vec{U}$ is ideal, as claimed.
The metric polarity $\widetilde{\Pi}$,  in this context, maps an equivalence class $[\vec{m}]$ to an ideal point perpendicular to the ideal point $\vec{U}$.  It maps all euclidean points (2-vectors) to the ideal line.

Equipped with this coordinate-free foundation of the algebra, the reader can now rejoin the article at \Sec{sec:gpdetail}.


\vspace{-.15in}

\end{document}